\documentclass[11pt,reqno]{amsart}
\usepackage{hyperref}
\usepackage{graphicx}
\usepackage{color}
\usepackage[all]{xy}
\usepackage{amsfonts,amssymb}
\usepackage{amsthm}
\usepackage{bbm} 

\usepackage{mathrsfs}

\usepackage{a4wide}

\newtheorem{neu}{}[section]

\newtheorem*{Cor*}{Corollary}
\newtheorem{Thm}[neu]{Theorem}
\newtheorem*{Thm*}{Theorem}

\newtheorem{Prop}[neu]{Proposition}
\newtheorem*{Prop*}{Proposition}
\theoremstyle{definition}
\newtheorem{Lemma}[neu]{Lemma}

\newtheorem*{Rmk*}{Remark}
\newtheorem{Rmk}[neu]{Remark}
\newtheorem{Ex}[neu]{Example}
\newtheorem*{Ex*}{Example}

\newtheorem*{Qu*}{Question}

\newtheorem{Def}[neu]{Definition}

\newcommand{\N}{\mathbb{N}}

\newcommand{\Z}{\mathbb{Z}}
\newcommand{\R}{\mathbb{R}}
\newcommand{\HH}{\mathbb{H}}

\newcommand{\wrt}{with respect to }

\newcommand{\CZ}{\mu_{\mathrm{CZ}}}

\newcommand{\id}{\mathrm{id}}
\newcommand{\ind}{\mathrm{ind\,}}

\newcommand{\om}{\omega}
\newcommand{\Om}{\Omega}

\newcommand{\acs}{almost complex structure}

\newcommand{\Hess}{\mathrm{Hess}}

\newcommand{\A}{\mathcal{A}}

\renewcommand{\P}{\mathcal{P}}

\newcommand{\D}{\mathbb{D}}

\newcommand{\M}{\mathcal{M}}

\renewcommand{\L}{\mathscr{L}}

\renewcommand{\H}{\mathrm{H}}

\newcommand{\CF}{\mathrm{CF}}

\newcommand{\FH}{\mathrm{FH}}

\newcommand{\RFH}{\mathrm{RFH}}
\newcommand{\RFHb}{\mathbf{RFH}}

\newcommand{\RFCb}{\mathbf{RFC}}

\newcommand{\CM}{\mathrm{CM}}
\newcommand{\HM}{\mathrm{HM}}

\newcommand{\Crit}{\mathrm{Crit}}

\newcommand{\beq}{\begin{equation}}
\newcommand{\beqn}{\begin{equation}\nonumber}
\newcommand{\eeq}{\end{equation}}

\newcommand{\bea}{\begin{equation}\begin{aligned}}
\newcommand{\bean}{\begin{equation}\begin{aligned}\nonumber}
\newcommand{\eea}{\end{aligned}\end{equation}}

\numberwithin{equation}{section}

\definecolor{Urs}{rgb}{0,.7,0}
\definecolor{Youngjin}{rgb}{0,0,1}
\definecolor{red}{rgb}{1,0,0}

\newcommand{\p}{\partial}
\newcommand{\Mp}{\mathfrak{M}}

\newcommand{\intinf}{{\int_{-\infty}^{\infty}}}
\newcommand{\<}{{\langle}}
\newcommand{\T}{{\mathbb{T}}}       

\begin{document}
\title[Continuation homomorphism in Rabinowitz Floer homology]
{Continuation homomorphism in Rabinowitz Floer homology for symplectic deformations}
\author{Youngjin Bae}
\author{Urs Frauenfelder}
\address{
    Youngjin Bae\\
    Department of Mathematics and Research Institute of Mathematics\\
    Seoul National University}
\email{jini0919@snu.ac.kr}
\address{
    Urs Frauenfelder\\
    Department of Mathematics and Research Institute of Mathematics\\
    Seoul National University}
\email{frauenf@snu.ac.kr}
\begin{abstract}
Will J.\,Merry computed Rabinowitz Floer homology above Ma\~n\'e's
critical value in terms of loop space homology in \cite{Mer10}  by
establishing an Abbondandolo-Schwarz short exact sequence. The
purpose of this article is to provide an alternative proof of
Merry's result. We construct a continuation homomorphism for
symplectic deformations which enables us to reduce the computation
to the untwisted case. Our construction takes advantage of a special
version of the isoperimetric inequality which above Ma\~n\'e's
critical value holds true.
\end{abstract}
\maketitle

\section{Introduction}

Rabinowitz Floer homology as introduced
in \cite{CF09} is the semi-infinite dimensional Morse homology
associated to Rabinowitz action functional. Critical points of
Rabinowitz action functional are Reeb orbits on a fixed energy
hypersurface of arbitrary period. Rabinowitz Floer homology vanishes
if the energy hypersurface is displaceable, however, we have the
following non-vanishing result.

\begin{Thm}[Abbondandolo-Schwarz \cite{AS09}, Cieliebak-Frauenfelder-Oancea \cite{CFO09}]\label{untwisted}
Assume $N$ is a closed manifold. Denote by $ST^*N$ the unit cotangent bundle of $N$
in the cotangent bundle $T^*N$ which is endowed with its canonical symplectic structure.
Then in degree $*\neq0,1$
\[
\RFHb_*(ST^*N,T^*N)=
\left\{
\begin{array}{ll}
\H_*(\L_N),&\text{if }\ *>1, \\
\H^{-*+1}(\L_N),&\text{if }\ *<0.
\end{array}
\right.
\]
If $e(T^*N)$ is the Euler characteristic of $T^*N$, then in degree 0
we have
\[
\RFH_0^c(ST^*N,T^*N)=
\left\{
\begin{array}{ll}
\H_0(\L^c_N)\oplus\H^1(\L^c_N),&\text{if }\ c\neq0, \\
\H_0(\L^0_N)\oplus\H^1(\L^0_N),&\text{if }\ c=0 \text{ and } e(T^*N)=0, \\
\H^1(\L^0_N), &\text{if }\ c=0 \text{ and } e(T^*N)\neq 0.
\end{array}
\right.
\]
In degree 1 we have
\[
\RFH_1^c(ST^*N,T^*N)=
\left\{
\begin{array}{ll}
\H_1(\L^c_N)\oplus\H^0(\L^c_N),&\text{if }\ c\neq0, \\
\H_1(\L^0_N)\oplus\H^0(\L^0_N),&\text{if }\ c=0 \text{ and } e(T^*N)=0, \\
\H_1(\L^0_N), &\text{if }\ c=0 \text{ and } e(T^*N)\neq 0.
\end{array}
\right.
\]
Here, $\L_N$ is the free loop space of $N$ and $\L^c_N$ is the
connected component of $\L_N$ of homotopy type $c$ and
$\RFH^c(ST^*N,T^*N)$ is the Rabinowitz Floer homology for the
Rabinowitz action functional restricted to $\L^c_{T^*N}$. Moreover,
all homology groups are taken with $\mathbb{Z}_2$-coefficients.
\end{Thm}

An interesting result of Will J. Merry tells us that this theorem
continuous to hold in the presence of a weakly exact magnetic field
for high enough energy levels. On the cotangent bundle $\tau:T^*N\to
N$ of a closed Riemannian manifold $(N,g)$, we consider an
autonomous Hamiltonian system defined by a convex Hamiltonian \beqn
H_U(q,p)=\frac{1}{2}|p|^2+U(q) \eeq and a twisted symplectic form
\beqn \om_{\sigma}=\om_0+\tau^*\sigma. \eeq Here $\om_0=dp\wedge dq$
is the canonical symplectic form in canonical coordinates $(q,p)$ on
$T^*N$, $|p|$ denotes the dual norm of a Riemannian metric $g$ on
$N$, $U:N\to\R$ is a smooth potential, and $\sigma$ is a closed
2-form on $N$. This Hamiltonian system describes the motion of a
particle on $N$ subject to the conservative force $-\nabla U$ and
the magnetic field $\sigma$. We call the symplectic manifold
$(T^*N,\om_{\sigma})$ a {\em twisted cotangent bundle}.

In order to state Will J. Merry's results we need the term of {\em
Ma\~n\'e critical value}. Let $(\widetilde{N},\widetilde g)$ be the
universal cover of $(N,g)$. Let $\sigma\in\Om^2(N)$ denote a closed
{\em weakly exact} 2-form, which means that the pullback
$\widetilde\sigma\in\Om^2(\widetilde N)$ is exact.

\begin{Def}
Let $\sigma\in\Om^2(N)$ be a closed weakly exact 2-form.
Then the {\em Ma\~n\'e critical value} is defined as
\[
c=c(g,\sigma,U):=\inf_{\theta\in\P_\sigma}\sup_{q\in\widetilde{N}}\widetilde{H}_U(q,\theta_q),
\]
where $\P_{\sigma}=\{\theta\in\Om^1(\widetilde{N})\;
|\;d\theta=\widetilde\sigma\}$ and $\widetilde{H}_U$ is the lift of
$H_U$ to the universal cover.
\end{Def}

In this article, we restrict our attention to the case of $c<\infty$
i.e. $\widetilde{\sigma}\in\Om^2(\widetilde{N})$ admits a bounded
primitive. For given $k\in\R$, we let $\Sigma_k:=H^{-1}_U(k)\subset
T^*N$. Then the dynamics of the hypersurface $\Sigma_k$ changes
dramatically when $k$ is passing through $c$. If $k>c$ then
$\Sigma_k$ is {\em virtual restricted contact}, and Rabinowitz Floer
homology is well-defined. All these things are investigated in
\cite{CFP09}. The following theorem was conjectured in \cite{CFP09}
and proved in \cite{Mer10} by using {\em the Abbondandolo-Schwarz
short exact sequence}.

\begin{Thm}[Merry \cite{Mer10}]\label{twisted}
Under the above assumptions if $k>c(g,\sigma,U)$, then in degree
$*\neq 0,1$
\[
\RFHb_*(\Sigma_k,T^*N,\om_\sigma)=
\left\{
\begin{array}{ll}
\H_*(\L_N),&\text{if }\ *>1, \\
\H^{-*+1}(\L_N),&\text{if }\ *<0.
\end{array}
\right.
\]
In degree $0,1$ we have the same result as in Theorem
\ref{untwisted}.
\end{Thm}
The aim of this article is to give an alternative proof of the above
theorem by constructing an explicit isomorphism between
$\RFHb(\Sigma_k,T^*N,\omega_0)$ and
$\RFHb(\Sigma_k,T^*N,\omega_\sigma)$ and then use the untwisted
version, namely Theorem~\ref{untwisted}. The explicit isomorphism is
given by the continuation homomorphism for the symplectic
deformation $r \mapsto \omega_{r\sigma}$ with $r \in [0,1]$. For the
following theorem note that $c(g,\sigma,U)\geq c(g,0,U)=\max U$.
Hence if $k>c(g,\sigma,U)$, then the Rabinowitz Floer homology for
$\Sigma_k$ is defined and coincides with the one from
Theorem~\ref{untwisted}.

\begin{Thm}\label{thm:rfhcon}
Under the above assumptions, if $k>c(g,\sigma,U)\geq c(g,0,U)=\max
U$ and $\om_0,\om_{\sigma}\in\Om^{\Mp}_{\rm reg}(\Sigma_k)$ then
there is a continuation map
\[
\Psi_{\om_0 *}^{\om_{\sigma}}:\RFCb_*(\Sigma_k,\om_0)\to\RFCb_*(\Sigma_k,\om_{\sigma})
\]
which induces an isomorphism
\[
\widetilde{\Psi_{\om_0}^{\om_{\sigma}}}_*:\RFHb_*(\Sigma_k,\om_0)\to\RFHb_*(\Sigma_k,\om_{\sigma}).
\]
\end{Thm}

One of our motivation for considering an alternative proof of
Merry's result is that the continuation homomorphism can be used to
compare spectral invariants between two different magnetic fields,
we refer to \cite{AF10} for a discussion of spectral invariants in
Rabinowitz Floer homology. We plan to discuss this in more detail in
a further paper.

The question of invariance under symplectic perturbation is also an
important issue in symplectic homology, we refer to the paper by
Ritter \cite{Rit}. In view of the long exact sequence between
symplectic homology and Rabinowitz Floer homology established in
\cite{CFO09} we expect interesting interactions of this paper with
the approach followed by Ritter.
\\ \\
\emph{Acknowledgement: } The authors were supported by the Basic
research fund 2010-0007669 funded by the Korean government.
We thank to J.-C. Sikorav for helpful comments
including Example \ref{ex:sol}.

\section{Continuation homomorphism in Morse and Floer homology}

\subsection{Morse homology}
Let $(M,g)$ be a closed Riemannian manifold and $f:M\to\R$ a Morse
function. We recall that the Morse chain complex $\CM_*(f)$ is the
graded $\Z_2$-vector space generated by the set $\Crit(f)$ of
critical points of $f$. The grading is given by the Morse index
$\mu=\mu_{\rm Morse}$ of $f$. The boundary operator
\[
\p:\CM_*(f)\to\CM_{*-1}(f)
\]
is defined on generators by counting gradient flow lines.
Indeed assume that a Riemannian metric $g$ on $M$ satisfies the following transversality condition.
Stable and unstable manifolds \wrt the negative gradient flow of $\nabla f=\nabla^g f$ intersect transversally,
that is, $W^s(x)\pitchfork W^u(y)$ for all $x,y\in\Crit(f)$.
Then the moduli space
\[
\widehat{\M}(x_-,x_+):=\{x:\R\to M\ |\ \p_sx(s)+\nabla f(x(s))=0,\ \lim_{s\to\pm\infty}x(s)=x_{\pm}\}
\]
is a smooth manifold of dimension $\dim\widehat{\M}(x_-,x_+)=\mu(x_-)-\mu(x_+)$.
Moreover, $\R$ acts by shifting the $s$-coordinate.
If $x_-\neq x_+$, the action is free and we denote the quotient by
\[
\M(x_-,x_+):=\widehat{\M}(x_-,x_+)/\R.
\]
Moreover, if $\mu(x_-)-\mu(x_+)=1$ then $\M(x_-,x_+)$ is a finite set.
Then we can define the differential $\p=\p(f,g)$ as a linear map which is given on generators by
\[
\p x_-:=\sum_{x_+\in\Crit(f)\atop\mu(x_-)-\mu(x_+)=1}\#_2\M(x_-,x_+)x_+,
\]
where, $\#_2$ denotes the count of a set modulo 2. It is a deep
theorem in Morse homology that the identity
\[
\p\circ\p=0
\]
holds, see \cite{Sch93} for details. Then
\[
\HM_*(f,g):=\H_*(\CM_\bullet(f),\p(f,g))
\]
is the Morse homology for the pair $(f,g)$.
Moreover, $\HM_*(f,g)$ equals the singular homology $\H_*(M)$ of $M$.
In particular, $\HM(f,g)$ is independent of the choice of Morse-Smale pair $(f,g)$.

The independence of $\HM(f,g)$ of Morse-Smale pair $(f,g)$ can be shown directly using the continuation homomorphism
which is constructed in the following way.
For two Morse-Smale pairs $(f_\pm,g_\pm)$ we choose $T>0$ and a smooth family $\{f_s,g_s\}_{s\in\R}$
of functions $f_s:M\to\R$ and Riemannian metrics $g_s$ such that
\bean
f_s=
\left\{
\begin{array}{l}
f_-\  \text{for} \ s\leq -T \\
f_+\  \text{for} \ s\geq T
\end{array}
\right.
\qquad
g_s=
\left\{
\begin{array}{l}
g_-\  \text{for} \ s\leq -T \\
g_+\  \text{for} \ s\geq T .
\end{array}
\right.
\eea
For critical points $x_\pm\in\Crit(f_\pm)$, we consider the moduli space
\[
\mathcal N(x_-,x_+)=\mathcal N(x_-,x_+;f_s,g_s):=\{x:\R\to M\ |\ \p_sx(s)+\nabla^{g_s}f_s(x(s))=0,\ \lim_{s\to\pm\infty}x(s)=x_\pm\}.
\]
A homotopy $(f_s,g_s)$ is called regular if the Fredholm operator
obtained by linearizing the gradient flow equation is onto. In
particular, for a regular homotopy the moduli space $\mathcal
N(x_-,x_+)$ is a smooth manifold of dimension $\dim\mathcal
N(x_-,x_+)=\mu(x_-)-\mu(x_+)$. A generic homotopy is regular.
Moreover, in the special case $f_s=f_-=f_+$ and $g_s=g_-=g_+$ we
have the identity \beq\label{eqn:NM} \mathcal
N(x_-,x_+)=\widehat{\M}(x_-,x_+). \eeq If $\mu(x_-)-\mu(x_+)=0$ the
space $\mathcal N(x_-,x_+)$ is compact. In order to verify that we
need to prove a uniform energy bound of $x\in\mathcal N(x_-,x_+)$ as
follows \bea\label{eqn:Meb}
E(x)&=E_{g_s}(x)=\intinf\|\p_s x(s)\|_{g_s}^2ds \\
&=-\intinf\langle\nabla_{g_s}f_s(x(s)),\p_s x(s)\rangle_{g_s}ds \\
&=-\intinf df_s(x(s))\p_s xds \\
&=-\intinf\frac{d}{ds}f_s(x(s))ds+\intinf\dot f_s(x(s))ds \\
&\leq\|f_-\|_{\infty}+\|f_+\|_{\infty}+2T\|\dot f_s\|_{\infty}.
\eea

Then we define a linear map
\bean
Z=Z(f_s,g_s):\CM_*(f_-)&\to \CM_*(f_+) \\
x_-&\mapsto\sum_{x_+\in\Crit(f_+)\atop\mu(x_-)=\mu(x_+)}\#_2\mathcal N(x_-,x_+)x_+.
\eea
We denote $\p_\pm:=\p(f_\pm,g_\pm)$.
In the same manner as $\p\circ\p=0$, one proves in Morse homology
\[
Z\circ\p_-=\p_+\circ Z,
\]
see \cite{Sch93}. In particular, on homology level we obtain the map
\[
\widetilde Z:\HM_*(f_-,g_-)\to\HM_*(f_+,g_+)
\]
which is called the continuation homomorphism.
By a homotopy-of-homotopies argument, it is proved that
$\widetilde Z$ is independent of the chosen homotopy $(f_s,g_s)$, see \cite{Sch93}.
Moreover, the continuation homomorphism is functorial in the following sense.
If we fix three Morse-Smale pairs $(f_a,g_a),\ (f_b,g_b)$, and $(f_c,g_c)$
we denote the corresponding continuation homomorphisms by $\widetilde Z_a^b:\HM_*(f_a,g_a)\to\HM_*(f_b,g_b)$
and similarly $\widetilde Z_a^c$ and $\widetilde Z_b^c$.
Then we have the following identity
\[
\widetilde Z_a^c=\widetilde Z_b^c\circ\widetilde Z_a^b.
\]
Now consider the case $f_s=f_a$ and $g_s=g_a$. By (\ref{eqn:NM}), we
get $\#_2\mathcal N(x_-,x_+)=1$ if $x_-=x_+$ and $\#_2\mathcal
N(x_-,x_+)=0$ for the other cases. Hence we obtain the following
identity
\[
\widetilde Z_a^a=\id_{\HM_*(f_a,g_a)}.
\]
In particular, we conclude that $\widetilde Z_a^b$ is an isomorphism with inverse $\widetilde Z_b^a$.

\subsection{Morse-Bott homology}\label{sec:MBh}

Let $M$ be a compact manifold and $(f,h,g,g^0)$ be a {\em Morse-Bott
quadruple}. The Morse-Bott quadruple consists of a Morse-Bott
function $f$ on M, a Morse function $h$ on $\Crit(f)$, a Riemannian
metric $g$ on $M$ and a Riemannian metric $g^0$ on $\Crit(f)$. We
assume that $(h,g^0)$ satisfies the Morse-Smale condition, i.e.
stable and unstable manifolds intersect transversally. For a
critical point $c$ on $h$, let $\ind_f(c)$ be the number of negative
eigenvalues of $\Hess(f)(c)$ and $\ind_h(c)$ be the number of
negative eigenvalues of $\Hess(h)(c)$. We define
\[
\ind(c):=\ind_{f,h}(c):=\ind_f(c)+\ind_h(c).
\]

\begin{Def}
For $c_1,c_2\in\Crit(h)$, and $m\in\N$ a {\em flow line from $c_1$
to $c_2$ with $m$ cascades}
\[
(x,T)=((x_k)_{1\leq k\leq m},(t_k)_{1\leq k \leq m-1})
\]
consist of $x_k\in C^\infty(\R,M)$ and $t_k\in\R_\geq:=\{r\in\R:r\geq0\}$ which satisfy the following conditions:
\begin{enumerate}
\item $x_k\in C^\infty(\R,M)$ are nonconstant solutions of
\[
\dot x_k=-\nabla f(x_k).
\]

\item There exists $p\in W^u_h(c_1)$ and $q\in W^s_h(c_2)$ such that
\[
\lim_{s\to-\infty}x_1(s)=p \text{ and } \lim_{s\to\infty}x_m(s)=q.
\]
\item For $1\leq k\leq m-1$ there are Morse flow lines $y_k\in C^\infty(\R,\Crit(f))$ of $h$,
i.e. solutions of
\[
\dot y_k=-\nabla h(y_k),
\]
such that
\[
\lim_{s\to\infty}x_k(s)=y_k(0),\qquad \lim_{s\to-\infty}x_{k+1}(s)=y_k(t_k).
\]
\end{enumerate}
A \emph{flow line with zero cascades} is just an ordinary Morse flow
line from $c_1$ to $c_2$.
\end{Def}

\begin{figure}
\centering
\includegraphics[width=0.45\textwidth]{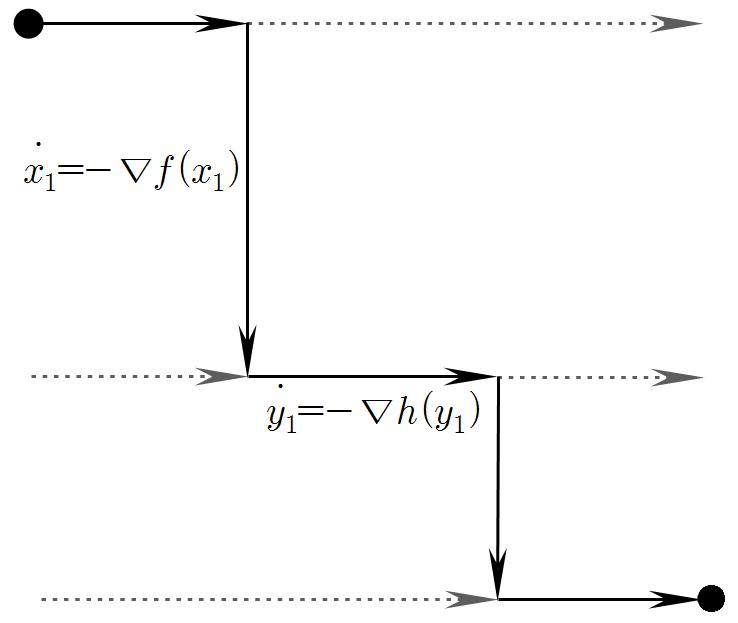}
\caption{A flow line with cascades}
\end{figure}

We denote the space of flow lines with $m$ cascades from $c_1$ to $c_2\in\Crit(h)$ by
\[
\widetilde\M_m(c_1,c_2).
\]
If $m\geq1$ then the group $\R^m$ acts on $\widetilde\M_m(c_1,c_2)$ by time shift on each cascade, i.e.
\[
x_k(s)\mapsto x_k(s+s_k).
\]
In the case of zero cascades $m=0$, the group $\R$ still acts on $\widetilde\M_0(c_1,c_2)$ by time shifting.
We denote the quotient by
\[
\M_m(c_1,c_2).
\]
We define the {\em set of flow lines with cascades from $c_1$ to $c_2$} by
\[
\M(c_1,c_2):=\bigcup_{m\in\N_0}\M_m(c_1,c_2).
\]

For a pair $(f,h)$ consisting of a Morse-Bott function $f$ on $M$ and
a Morse function $h$ on $\Crit(f)$, we define the chain complex $\CM_*(f,h)$
as the $\Z_2$-vector space generated by the critical points of $h$ graded by the index.
More precisely, $\CM_k(f,h)$ are formal sums of the form
\[
\xi=\sum_{c\in\Crit(h)\atop\ind(c)=k}\xi_cc
\]
with $\xi_c\in\Z_2$. For generic choice of the Riemannian metric $g$ on $M$,
the moduli spaces of flow lines with
cascades $\M(c_1,c_2)$ is a smooth manifold of dimension
\[
\dim\M(c_1,c_2)=\ind(c_1)-\ind(c_2)-1.
\]
If $\dim\M(c_1,c_2)=0$, then $\M(c_1,c_2)$ is finite. We define the boundary operator
\[
\p_k:\CM_k(f,h)\to\CM_{k-1}(f,h)
\]
as the linear extension of
\[
\p_kc=\sum_{\ind(c')=k-1}\#_2\M(c,c')c'
\]
for $c\in\Crit(h)$ with $\ind(c)=k$. The usual gluing and compactness arguments imply that
\[
\p\circ\p=0.
\]
This defines homology groups
\[
\HM_*(f,h,g,g^0):=\H_*(\CM_\bullet(f,h),\p(f,h,g,g^0)).
\]

In the Morse-Bott situation, we can also show that the Morse-Bott
homology is independent of the choice of the Morse-Bott quadruple.
First take two regular quadruples $(f_-,h_-,g_-,g^0_-)$ and
$(f_+,h_+,g_+,g^0_+)$. Choose a smooth family of interpolations
$\{(f_s,g_s)\}_{s\in\R}$ such that \bean f_s= \left\{
\begin{array}{l}
f_-\  \text{for} \ s\leq -T \\
f_+\  \text{for} \ s\geq T
\end{array}
\right.
\qquad
g_s=
\left\{
\begin{array}{l}
g_-\  \text{for} \ s\leq -T \\
g_+\  \text{for} \ s\geq T .
\end{array}
\right.
\eea

For $c_1\in\Crit(h_-)$, $c_2\in\Crit(h_+)$,
we consider the following flow lines from $c_1$ to $c_2$ with $m$ cascades
\[
(x,T)=((x_k)_{1\leq k\leq m},(t_k)_{1\leq k\leq m-1})
\]
for $x_k\in C^\infty(\R,M)$ and $t_k\in\R_\geq$ which satisfy the following conditions:
\begin{enumerate}
\item $x_k$ are solutions of
\[
\dot x_k(s)=-\nabla_{\widetilde g_{k}}\widetilde f_k(x_k),
\]
where for some $m_1 \in \{1,\ldots,m\}$ \beqn \widetilde f_k=
\left\{
\begin{array}{ccc}
f_-&\text{ for }&1\leq k\leq m_1-1 \\
f_s&\text{ for }& k=m_1 \\
f_+&\text{ for }& m_1+1\leq k\leq m
\end{array}
\right.
\eeq
and
\beqn
\widetilde g_k=
\left\{
\begin{array}{ccc}
g_-&\text{ for }&1\leq k\leq m_1-1 \\
g_s&\text{ for }& k=m_1 \\
g_+&\text{ for }& m_1+1\leq k\leq m.
\end{array}
\right. \eeq Moreover, for $k \neq m_1$ the cascade $x_k$ is
nonconstant.

\item There exists $p_1\in W^u_{h_-}(c_1)$ and $p_2\in W^s_{h_+}(c_2)$ such that
$\lim_{s\to-\infty}x_1(s)=p_1$ and $\lim_{s\to\infty}x_m(s)=p_2$.

\item
For $1\leq k\leq m-1$, $y_k$ are Morse flow lines of $\widetilde h$, i.e. solutions of
\[
\dot y_k(s)=-\nabla_{\widetilde g^0_{k}}\widetilde h_k(y_k),
\]
and
\[
\lim_{s\to\infty}x_k(s)=y_k(0),\qquad \lim_{s\to-\infty}x_{k+1}(s)=y_k(t_k)
\]
where
\beqn
\widetilde h_k=
\left\{
\begin{array}{ccc}
h_-&\text{ for }&1\leq k\leq m_1-1 \\
h_+&\text{ for }& m_1\leq k\leq m-1
\end{array}
\right.
\eeq
and
\beqn
\widetilde g^0_k=
\left\{
\begin{array}{ccc}
g^0_-&\text{ for }&1\leq k\leq m_1-1 \\
g^0_+&\text{ for }& m_1\leq k\leq m-1.
\end{array}
\right.
\eeq
\end{enumerate}

For a generic choice of the data, the space of solutions of (1) to
(3) is a smooth manifold whose dimension is given by the difference
of the indices of $c_1$ and $c_2$. If $\ind(c_1)=\ind(c_2)$ then
this manifold is compact. In order to verify this we need to prove a
uniform energy bound of time-dependent cascades as in the Morse
case. Since a cascade consists of several negative gradient flow
lines $(x_k)_{1\leq k\leq m},(y_k)_{1\leq k\leq m-1}$, it suffices
to show that the energy of each (time-dependent) gradient flow line
are uniformly bounded. This is guaranteed by the argument of
(\ref{eqn:Meb}) in the Morse situation.

We define a map
\[
Z=Z(\widetilde f,\widetilde h,\widetilde g,\widetilde g^0):\CM_*(f_-,h_-)\to\CM_*(f_+,h_+)
\]
as the linear extension of
\[
Zc_-=\sum_{c_+\in\Crit(h_+)\atop\ind(c_+)=\ind(c_-)}\#_2\M(c_-,c_+)c_+
\]
where $c_-\in\Crit(h_-)$. Standard arguments as in the Morse case
show that $Z$ induces isomorphisms on homologies
\[
\widetilde Z:\HM_*(f_-,h_-,g_-,g^0_-)\to\HM_*(f_+,h_+,g_+,g^0_+).
\]
This proves that Morse-Bott homology is independent of the choice of
a Morse-Bott quadruple. We refer to Appendix A in \cite{Fra04}, for
details.

\subsection{Floer homology for Hamiltonian deformation}
Let $(M,\om)$ be a symplectically aspherical closed $2n$-dimensional manifold
which means that $\om|_{\pi_2(M)}\equiv 0$.
Let $H:S^1\times M\to\R$ be a time-dependent Hamiltonian on $M$ and $H_t=H(t,\cdot)\in C^\infty(M,\R)$.
The {\em Hamiltonian vector field} $X_{H_t}$ is defined by
\[dH_t = -\iota_{X_{H_t}}\om.\]
An almost complex structure $J_t$ on $M$ is $\om$-{\em compatible}
if $\langle\cdot,\cdot \rangle :=\om(\cdot,J_t\cdot)$ is a Riemannian metric $\forall t\in S^1$.
Let $\L^0$ be the component of contractible loops on $M$.
The {\em Hamiltonian action} is
\[
\A_H:\L^0\to\R
\]
\[\A_H(x):=\int_{\D^2} \overline{x}^*\om-\int_0^1H(t,x(t))dt,\]
where $\overline{x}$ is an extension of $x$ to the unit disk. Since
we consider only contractible loops such an extension exists and
because $\om|_{\pi_2(M)}=0$ the action functional does not depend on
the choice of the filling disk. A positive gradient flow line
$v:\R\times S^1 \to M$ of $\A_H$ satisfies the perturbed
Cauchy-Riemann equation \beq\label{eq:grad}
\p_sv+J(t,v)(\p_tv-X_H(t,v))=0. \eeq Formally a positive gradient
flow line $v\in$``$C^\infty(\R,\L^0)$'' is a solution of the ``ODE''
\[
\p_sv-\nabla\A_H(v(s))=0.
\]
According to Floer, we interpret this as a solution of the PDE,
$v\in C^\infty(\R\times S^1,M)$ satisfying (\ref{eq:grad}).
\subsubsection{Sign and grading conventions}
The {\em Conley-Zehnder index} $\CZ(x;\tau) \in \Z$ of a nondegenerate 1-periodic orbit $x$ of $X_H$ with respect to a
symplectic trivialization $\tau : x^*TM \to S^1 \times \R^{2n}$ is defined as follows. The linearized Hamiltonian flow
along $x$ with $\tau$ defines a path of symplectic matrices $\Phi_t,\ t \in [0,1]$, with $\Phi_0=\id$ and $\Phi_1$ not having 1
as its spectrum. Then $\CZ(x;\tau)$ is the Maslov index of the path $\Phi_t$ as defined in \cite{RS,Sal}.
For a critical point $x$ of a $C^2$-small Morse function $H$ the Conley-Zehnder index
with respect to the constant trivialization $\tau$ is related to the Morse index by
\[\CZ(x;\tau)=n-\mu_{\rm Morse}(x).\]
We have the following identity
\[\CZ(x;\tau')=\CZ(x;\tau)-2c_1([\tau' \# \overline{\tau}]),\]
where $c_1$ is first Chern class and $\overline{\tau}$ means
opposite orientation of $\tau$. If $c_1(M)=0$ we obtain integer
valued Conley-Zehnder indices for all 1-periodic orbits. Without any
hypothesis on $c_1(M)$ we still have well-defined Conley-Zehnder
indices in $\Z_2$ and all the following result hold with respect to
this $\Z_2$-grading.

\subsubsection{Floer homology}
Let $\P(H)$ be the set of 1-periodic orbits of $X_H$. Given
$x_\pm\in\P(H)$ we denote by $\widehat{\M}(x_-,x_+)$ the space of
solutions of (\ref{eq:grad}) with $\lim_{s \to \pm
\infty}v(s,t)=x_{\pm}(t)$. Its quotient by the $\R$-action
$s_0\cdot(s,t):=(s+s_0,t)$ on the cylinder is called the {\em moduli
space of Floer trajectories} and is denoted by
\[\M(x_-,x_+):=\widehat{\M}(x_-,x_+)/\R\]
Assume now that all elements of $\P(H)$ are nondegenerate.
Suppose further that the almost complex structure $J=(J_t),\;t\in S^1$ is generic,
so that $\M(x_-,x_+)$ is a smooth manifold of dimension
\[\dim\M(x_-,x_+)=\CZ(x_-)-\CZ(x_+)-1 .\]
The boundary operator $\p_k:\CF_k(H)\to\CF_{k-1}(H)$ is defined by
\[\p x:=\sum_{\CZ(y)=k-1}\#_2\M(x,y)y.\]
It increases the action and satisfies $\p\circ\p=0$.
Hence we can define Floer homology
\[
\FH_*(H)=\H_*(\CF_\bullet(H),\p).
\]
Note that $(\CF_*(H),\p)$ depends on additional data, namely the Hamiltonian $H$, the symplectic structure $\om$,
and the almost complex structure $J_t$.

\subsubsection{Continuation map}
Floer proved that $\FH_*(H)$ depends only on the underlying manifold
$M$, see \cite{Fl1,Fl2,Fl3}. We now give a proof of Floer's theorem
via Morse-Bott methods which is due to Piunikhin, Salamon and
Schwarz \cite{PSS}. Take two different time-dependent Hamiltonians
$H_-,H_+\in C^\infty(S^1\times M)$, and choose $T>0$ and a smooth
family of Hamiltonians $H_s:S^1\times M\to\R$ with $s\in\R$ such
that \beqn H_s= \left\{
\begin{array}{l}
H_-\  \text{for} \ s\leq -T \\
H_+\  \text{for} \ s\geq T .
\end{array}
\right.
\eeq
Now take two different almost complex structures $J_{t,-},J_{t,+}$,
and a smooth family of almost complex structures $J_{t,s}$ such that
\beqn
J_{t,s}=
\left\{
\begin{array}{l}
J_{t,-}\  \text{for} \ s\leq -T \\
J_{t,+}\  \text{for} \ s\geq T .
\end{array}
\right.
\eeq
The continuation map between two different time-dependent Hamiltonian
\[
\zeta_{H_-}^{H_+}:\CF_*(H_-)\to\CF_*(H_+)
\]
is given by counting positive gradient flow lines $v\in$``$C^\infty(\R,\L^0)$'' of
\[
\A_{H_s}(x)=\int_{\D^2}\overline{x}^*\om-\int_0^1 H_s(t,x(t))dt
\]
where, $v\in C^\infty(\R\times S^1,M)$ is a solution of
\beq\label{eqn:fhgrad}
\left.
\begin{array}{r}
\p_sv+J_{t,s}(v)(\p_t v-X_{H_s}(v))=0 \\
\lim_{s\to\pm\infty}v(s)=v_\pm\in\Crit\A_{H_\pm}.
\end{array}
\right\}
\eeq
For critical points $v_\pm\in\Crit\A_{H_\pm}$, we consider the moduli spaces
\[
\mathcal N_{H_\pm}(v_-,v_+)=\mathcal N_{H_\pm}(v_-,v_+;H_s)=\{v:\R\times S^1\to M\ |\ v\text{ satisfies }(\ref{eqn:fhgrad})\}.
\]
If $\CZ(v_-)=\CZ(v_+)$ the space $\mathcal N_{H_\pm}(v_-,v_+)$ is
compact. A crucial ingredient for the compactness proof is, as in
the Morse case, a uniform energy bound for $v\in\mathcal
N_{H_\pm}(v_-,v_+)$, see \cite{Sal} for details. The uniform energy
bound is given by \bean
E(v)&=E_{J_{t,s}}(v)=\intinf\|\p_s v(s)\|_{J_{t,s}}^2ds \\
&=\intinf\langle\nabla^{J_{t,s}}\A_{H_s}(v(s)),\p_s v(s)\rangle_{J_{t,s}} ds \\
&=\intinf\frac{d}{ds}\A_{H_s}(v(s))ds-\intinf\dot\A_{H_s}(v(s))ds \\
&=\A_{H_+}(v_+)-\A_{H_-}(v_-)-\intinf\int_0^1\dot H_s(t,v(s,t))dtds \\
&\leq\A_{H_+}(v_+)-\A_{H_-}(v_-)+2T\max_{s\in[-T,T]\atop (t,x)\in S^1\times M}|\dot H_s(t,x)| ,
\eea
where $\|\cdot\|_{J_{t,s}}$ is given by $\int_0^1\om(\cdot,J_{t,s}\cdot)dt$.
Then we can define a linear map
\bean
\zeta_{H_-}^{H_+}=\zeta_{H_-}^{H_+}(H_s):\CF_*(H_-)&\to\CF_*(H_+) \\
v_-&\mapsto\sum_{v_+\in\Crit\A_{H_+}\atop\CZ(v_-)=\CZ(v_+)}\#_2\mathcal
N_{H_\pm}(v_-,v_+)v_+ \eea which induces a homomorphism on homology
level,
\[
\widetilde\zeta_{H_-}^{H_+}:\FH_*(H_-)\to\FH_*(H_+).
\]
The resulting homomorphism is independent of the choice of the homotopy $H_s$ and $J_{t,s}$
by a homotopy-of-homotopies argument, similar as in the Morse situation.
By functoriality, we conclude that $\widetilde\zeta_{H_-}^{H_+}$ is an isomorphism with inverse $\widetilde\zeta_{H_+}^{H_-}$.

Now consider the special case, where Hamiltonian $H\equiv 0$ is the zero Hamiltonian.
Then
\[
\A_H(x)=\A_0(x)=\int_{\D^2}\overline x^*\om
\]
is the symplectic area functional which is Morse-Bott and
\[
\Crit\A_0=\{x\in\L^0\ |\ x\text{ is a constant loop}\}\cong M.
\]
This implies that
\[
\FH_*(0,f)=\HM_*(f)\cong \H_*(M),
\]
where $f:M\to\R$ is an additional Morse function on the critical manifold $\Crit\A_0\cong M$.
Note that $\H_*(M)$ is the singular homology of $M$ which only depends on $M$.
Hence we conclude that Floer homology does not depend on additional structures like $\om,H$, and $J_t$.

\subsection{Floer homology for symplectic deformation}
In the previous subsection, we have seen that Floer homology is
independent of the symplectic structure. In this subsection, we ask
if this fact can also be seen directly by constructing a
continuation homomorphism between two symplectic forms. So far we
can only construct the continuation homomorphism for symplectic
deformations under additional assumptions on the symplectic
structures. Different from the case of Hamiltonian deformations, it
might be necessary to subdivide the symplectic deformations in a
sequence of small {\em adiabatic} steps.

Let $(M,g)$ be a $2n$-dimensional closed Riemannian manifold with two symplectic forms $\om_0$, $\om_1$.
Suppose that $(M,\om_s)$ is a family of symplectically aspherical closed manifolds,
where $\om_s=s\om_1+(1-s)\om_0$ for $s\in[0,1]$.
Then we want to construct a continuation map
\beqn
\Psi_{\om_0*}^{\om_1}:\CF_{*}(\om_0)\to\CF_{*}(\om_1)
\eeq
which induces an isomorphism on homology level.
In order to state and prove our result we need the term of the {\em cofilling function}.

\begin{Def}[Gromov \cite{Gro1}, Polterovich \cite{Pol}]\label{def:cofilling}
Let $\sigma\in\Om^2(M)$ be a closed weakly exact 2-form, then the {\em cofilling function} is
\[
u_{\sigma}(s):[0,\infty)\to[0,\infty)
\]
\[
u_{\sigma}(s)=u_{\sigma,g,x}(s)=\inf_{\theta\in\P_\sigma}\sup_{z\in
B_x(s)}|\theta_z|_{\widetilde{g}},
\]
where $\P_{\sigma}=\{\theta\in\Om^1(\widetilde{M})\;
|\;d\theta=\widetilde\sigma\}$ is the space of primitives for
$\sigma$ and $B_x(s)$ be the $s$-ball centered at
$x\in\widetilde{M}$.
\end{Def}

\begin{Rmk}\label{rmk:diffmetric}
If we choose another Riemannian metric $g'$ on $M$ and a different
base point $x'\in\widetilde{M}$ then we can check that
$u_{\sigma,g,x}\sim u_{\sigma,g',x'}$. \footnote{ $f\sim g \iff
f\lesssim g$ and $\ g\lesssim f$,\qquad $f\lesssim g \iff \exists$
$C>0$ such that $f(s)\leq C(g(s)+1),\ \forall s\in[0,\infty)$. }
Moreover, since the function $u_{\sigma,g,x}$ actually only depends
on the projection of $x$ from the universal cover to the compact
space $M$, the constant can be chosen uniformly in $x$.
\end{Rmk}

\begin{Ex}\label{ex:torus}
Let us consider $(\T^{2n}=\R^{2n} \slash \Z^{2n}, \om=\sum_{i=1}^{n}dx_i \wedge dy_i)$
with the metric induced by the standard metric on $\R^{2n}$.
Since $\om=d(\sum_{i=1}^{n}x_i \wedge dy_i)$ and
$\sum_{i=1}^{n}x_i \wedge dy_i$ has {\em linear growth} on the universal cover $\R^{2n}$,
it follows that $u(s) \lesssim s$.

Suppose that there is $\theta\in\Om^1(\R^{2n})$ such that $d\theta=\om$ and
\[
\sup_{z\in B_x(s)}|\theta_z|\leq Cs^\alpha
\]
for $0\leq\alpha<1$, then we get
\[
\pi r^2=\int_{\D_r}d\theta=\int_{\partial\D_r}\theta
\leq \max_{z\in\partial \D_r}|\theta_z|\int_{\partial \D_r}1
\leq Cr^\alpha \ 2\pi r,
\]
where $\D_r$ is a 2-dimensional disk of radius $r$.
This cannot happen as $r\to\infty$, thus we conclude that $u_\om(s)\sim s$.
\end{Ex}

\begin{Ex}\label{ex:hyperbolic}
Now consider $(\HH^2, \om=\frac{1}{y^2}dx \wedge dy)$
with hyperbolic metric $ds^2=\frac{1}{y^2}(dx^2+dy^2)$. The given symplectic form $\om$
has a {\em bounded} primitive 1-form $\frac{1}{y}dx$ which means that $u(s) \sim 1$.
It is well-known that a bounded 2-form on $\HH^n$ with canonical hyperbolic metric has
constant cofilling function, see Gromov \cite[$5.B_5$]{Gro1}.
\end{Ex}

\begin{Ex}[Solvable manifold]\label{ex:sol}
Let us construct a 3-manifold $M$ fibered over $S^1$ with fiber $\T^2$ with hyperbolic monodromy,
\[
A=\begin{pmatrix}
2&1\\
1&1
\end{pmatrix}.
\]
Let $y,z$ be the coordinates of the fiber torus, then $\sigma =dy \wedge dz$ is a well-defined 2-form on $M$.
Note that every primitive of $\sigma$ on the universal cover
has exponential growth, see Appendix \ref{app:sol}.

\end{Ex}

\begin{Lemma}[Quadratic isoperimetric inequality]\label{thm:isoperimetric}
Let $(M,g)$ be a closed Riemannian manifold with closed weakly exact 2-form $\sigma\in\Om^2(M)$.
If $u_{\sigma}(t) \lesssim t$, then the {\em quadratic isoperimetric inequality} holds,
\beqn
\int_{\D^2}\overline{v}^*\sigma\leq C\left(l(v)^2+1\right)
\eeq
where $l(v)=\int_{S^1}|\p_tv(t)|_gdt$,
$\overline v:\D^2\to M$ is an extension of the contractible loop $v:S^1\to M$ to the unit disk, and $C=C(M,g,\sigma)$.
\end{Lemma}

\begin{proof}
Since $u_{\sigma}(t) \lesssim t$, we can choose a 1-form
$\theta\in\P_{\sigma}$ which has linear growth on the universal
cover and such that $\max_{z\in B_{\widetilde v(0)}(l(\widetilde
v))}|\theta_z|_{\widetilde g} \leq u_\sigma(l(\widetilde v))+1$. Let
$\widetilde{\overline{v}}\ : \D^2 \to \widetilde{M}$ be the lifting
of $\overline{v}$ and set $\theta_{\max}(\widetilde v) = \underset{z
\in \widetilde v(S^1)}{\max} |\theta_z|_{\widetilde g}$. Note that
\bean
\theta_{\max}(\widetilde v)&=\max_{z\in\widetilde v(S^1)}|\theta_z|_{\widetilde g} \\
&\leq\max_{z\in B_{\widetilde v(0)}(l(\widetilde v))}|\theta_z|_{\widetilde g} \\
&\leq u_\sigma(l(\widetilde v))+1 \\
&\leq \frac{C}{2}(l(\widetilde v)+1),
\eea
for some $C=C(M,g,\sigma)\in\R^+$.
The last inequality uses the fact $u_\sigma(t)\lesssim t$.
Then we get
\bean
\int_{\D^2}\overline{v}^*\sigma&=\int_{\D^2}\widetilde{\overline{v}}^* \widetilde{\sigma}\\
&= \int_{\D^2}\widetilde{\overline{v}}^* d\theta\\
&= \int_{S^1}\widetilde v^* \theta\\
&\leq\theta_{\max}(\widetilde v)l(v)\\
&\leq \frac{C}{2}\left(l(v)^2+l(v)\right) \\
&\leq C\left(l(v)^2+1\right).
\eea
Let us denote the constant $C$ as the {\em isoperimetric constant}.
\end{proof}


\begin{Def}
Let $M^{2n}$ be a closed manifold with a time-dependent Hamiltonian
$H:S^1\times M\to\R$. A pair $(\om_0,\om_1)$ is called a {\em
continuation pair} on $(M,H)$ if
\begin{itemize}
\item $(M,\om_s)$ is a symplectically aspherical closed manifold $\forall s\in[0,1]$, \\
where $\om_s=\om_0+s\sigma,\ \ \sigma=\om_1-\om_0$;
\item $\A_{\om_s}=\A_{H,\om_s}:\L^0\to\R$ is Morse, for generic $s\in[0,1]$ and $s=0,1$;
\item $u_\sigma(t)\lesssim t$.
\end{itemize}
\end{Def}

\begin{Rmk}\label{rmk:lip_prop}
Let us apply Lemma \ref{thm:isoperimetric} to the {\em continuation pair} $(\om_0,\om_1)$ on $M$.
First set
\beqn\label{eqn:lin_symp_form}
\om_s=\om_0+\beta(s)\sigma,\qquad \sigma=\om_1-\om_0
\eeq
where $\beta(s)\in C^\infty(\R,[0,1])$ is a cut-off function satisfying $\beta(s)=1$ for $s \geq 1$,
$\beta(s)=0$ for $s \leq 0$ and $0\leq\dot{\beta}(s) \leq 2$. Then we get
\bean
\left| \int_{\D^2}\overline{v}^*(\om_s-\om_0) \right|
&\leq \left| \int_{\D^2}\overline{v}^*\beta(s)\sigma \right|
= \beta(s)\left| \int_{\D^2}\overline{v}^*\sigma \right| \\
&\leq C\beta(s)\left[\left( \int_{S^1} |\partial_t v(t)| dt \right)^2+1\right].
\eea
Note that $C\beta(s)$ is continuous and $C\beta(s)=0$ for $s\leq0$.
Let us denote the function $C\beta(s)$ as the {\em isoperimetric constant function}.
\end{Rmk}

\begin{Thm}\label{thm:continuation}
Let $M^{2n}$ be a closed manifold with a time-dependent Hamiltonian $H:S^1\times M\to\R$.
If $(\om_0,\om_1)$ is a continuation pair on $(M,H)$ then there is a continuation map
\[
\Psi_{\om_0*}^{\om_1}:\CF_*(\om_0)\to\CF_*(\om_1)
\]
which induces an isomorphism
\[
\widetilde{\Psi_{\om_0}^{\om_1}}_*:\FH_*(\om_0)\to\FH_*(\om_1).
\]
\end{Thm}
\begin{proof}
First recall the definition of the action functional
\[
\A_{H,\om}:\L^0\to\R
\]
\beqn \A_{H,\om}(x)=\int_{\D^2}\overline{x}^*\om-\int_0^1H(t,x(t))dt,
\eeq where $\overline{x}:\D^2\to M$ is an extension of the
contractible loop $x$ to the unit disk. By the Morse condition in
the definition of the continuation pair $(\om_0,\om_1)$, we know
that $\om_0,\om_1\in\Om^{\rm symp}(M)$ are nondegenerate symplectic
forms. This means that every fixed point
$x\in\text{Fix}\phi_{H,\om_i}^1$ is nondegenerate, where
$\phi_{H,\om_i}^1:M\to M$ is the time-1-map for the flow of the
non-autonomous Hamiltonian vector field $X_H^{\om_i}$.

Let us consider
\[\om_s=\om_0+\beta(s)\sigma,\qquad \sigma=\om_1-\om_0\]
as in Remark \ref{rmk:lip_prop}. We choose further almost complex
structure $J_{s,t}$ for $\om_s$. For technical reasons, we now
subdivide $\om_s$ into sufficiently small pieces. Let
$\{\om^i\}_{i=0}^{N}$ be a subdivision of $\om_s$ satisfying
\begin{itemize}
\item $\om^i=\om_0+d(i)\sigma$, where $0=d(0)<d(1)<\cdots <d(N)=1$;
\item $\A_{H,\om^i}:\L^0\to\R$ is Morse, $\forall i=0,1,\dots,N$;
\item $C(M,g,(d(i+1)-d(i))\sigma)\leq 1/8$, $\forall i=0,1,\dots,N-1$,\\ where $C$ is the isoperimetric constant.
\end{itemize}
The above 2nd condition is guaranteed by the generic Morse condition
for the continuation pair $(\om_0,\om_1)$. By Remark
\ref{rmk:lip_prop}, we can assume the 3rd condition.

Let $\om_s^i=\om^i+\beta(s)(\om^{i+1}-\om^i)$ be a homotopy between $\om^i$ and $\om^{i+1}$.
Now consider $v:\R \times S^1 \to M$ satisfying the gradient flow equation
\beq\label{eqn:gradeqn}
\partial_sv+J_{s,t}(v)(\partial_tv-X_H^{\om_s^i}(t,v))=0,
\eeq and the limit condition \beq\label{eqn:limit_of_grad}
    \lim_{s \to -\infty}v(s,t)=v_- (t) \in \Crit \A_{H,\om^i}
    \quad
    \lim_{s \to +\infty}v(s,t)=v_+ (t) \in \Crit \A_{H,\om^{i+1}}.
\eeq
We then want to define a map
\[
\Psi_{\om^i \ \ k}^{\om^{i+1}}:\CF_k(\om^i)\to\CF_k(\om^{i+1})
\]
given by
\beqn
\Psi_{\om^i\ \ k}^{\om^{i+1}}(v_-)=\sum_{\CZ(v_+)=k}\#_2\M_{v_-,v_+}(\om^i,\om^{i+1})v_+.
\eeq
Here,
\beqn
\M_{v_-,v_+}(\om^i,\om^{i+1})=\{v:\R\times S^1\to M\ |\ v\text{ satisfies }(\ref{eqn:gradeqn}),\ (\ref{eqn:limit_of_grad})\}.
\eeq

Because $\om_s$ is symplectically aspherical $\forall s\in\R$, there is no bubbling.
So it suffices to bound the energy $E(v)=\intinf\|\partial_sv\|_{s}^2ds$
of $v \in C^{\infty}(\R \times S^1,M)$ in terms of $v_-,\ v_+$
where, $\|\cdot,\cdot\|_s$ is the $L^2$-norm defined by $\int_0^1\om_s(\cdot,J_s\cdot)dt$.
We first compute
\bean
E(v)&=\int_{-\infty}^{\infty} \| \partial_s v \|_{s}^2 ds \\
&=\intinf\<\partial_s v, \nabla \A_{H,\om_s^i}(v) \rangle_{s}ds\\
&=\intinf\frac{d}{ds}\A_{H,\om_s^i}(v)ds-\intinf\dot{\A}_{H,\om_s^i}(v)ds \\
&=\A_{H,\om^{i+1}}(v_+)-\A_{H,\om^i}(v_-)-\intinf\dot{\A}_{H,\om_s^i}(v)ds.
\eea
So we need to consider the following
\bean
\left| \int_{-\infty}^{\infty}\dot{\A}_{H,\om_s^i}(v)ds \right|
&\leq \int_{\infty}^{\infty}\dot{\beta}(s) \left| \int_{\D^2}\overline{v}^*(\om^{i+1}-\om^i) \right| ds\\
&\leq \intinf \dot{\beta}(s) C\left( \int_{S^1}|\partial_{t}v|_s dt \right)^2 ds+C
\eea
For some $C=C(M,g,(d(i+1)-d(i))\sigma)$.
Here $|\cdot,\cdot|_s$ is the norm on $M$ induced by the Riemannian metric $\om_s(\cdot,J_s\cdot)$.
From the equation (\ref{eqn:gradeqn}), we get
\beq\label{eqn:gradeqn2}
\partial_t v = J(s,v)\partial_s v + X_H^{\om_s^i} (v).
\eeq \\
By putting the above equation (\ref{eqn:gradeqn2}) into the isoperimetric inequality, we obtain
\bean\label{eqn:energybound1}
\left| \int_{-\infty}^{\infty}\dot{\A}_{H,\om_s^i}(v)ds \right|
&\leq\intinf \dot{\beta}(s) C\left( \int_{S^1}|\partial_{t}v|_s dt \right)^2 ds+C \\
&\leq C \intinf \dot{\beta}(s) \|\partial_t v \|_s^2 ds+C \\ 
&= C \intinf \underbrace{\dot{\beta}(s)}_{\leq 2} \<J_s(v)\partial_s v + X_H^{\om_s^i}(v),J_s(v)\partial_s v + X_H^{\om_s^i}(v) \rangle_s ds+C\\
&\leq 2C \left( \int_{0}^{1} \| \partial_s v \|_s^2 ds+
\int_{0}^{1} \underbrace{2\<J_s\partial_sv,X_H^{\om_s^i}(v) \rangle_s}_{\leq \| \partial_s v \|_s^2+\| X_H^{\om_s^i}(v) \|_s^2} ds +
\int_{0}^{1} \| X_H^{\om_s^i}(v) \|_s^2 ds \right)+C\\
&\leq 4C \intinf \|\partial_s v\|_s^2 ds + 4C \int_{0}^{1}\|X_H^{\om_s^i}(v)\|_s^2ds+C \\
&\leq 4C\;E(v) + 4C\;c'+C,
\eea
where $c' \in \R$ is chosen satisfying $\|X_H^{\om_s^i}(v)\|_s^2 \leq c'$. This is possible by the compactness of $M$. Thus we get
\bean\label{eqn:energybound2}
E(v) \leq \underbrace{\A_{H,\om^{i+1}}(v_+)-\A_{H,\om^i}(v_-)+4Cc'+C}_{=:c''} +4C\;E(v).
\eea
Since $C=C(M,g,(d(i+1)-d(i))\sigma) \leq \frac{1}{8}$, we finally obtain
\bean\label{eqn:energybound3}
E(v) &\leq c'' + \frac{1}{2}E(v) \\
E(v) &\leq 2c''.
\eea

Now we define the continuation map from $\om_0$ to $\om_1$ by juxtaposition
\[
\Psi_{\om_0}^{\om_1}:\CF(\om_0)\to\CF(\om_1)
\]
\[
\Psi_{\om_0}^{\om_1}=\Psi_{\om^{N-1}}^{\om^N}\circ\cdots\circ\Psi_{\om^1}^{\om^2}\circ\Psi_{\om^0}^{\om^1}.
\]
By a standard argument in Floer homology theory,
each $\Psi^{\om^{i+1}}_{\om^i}$ commutes with the boundary operators of the Floer chain complex.
This implies that $\Psi_{\om_0}^{\om_1}$ also interchanges the boundary operators.
Hence we get an induced homomorphism
\beqn
\widetilde{\Psi_{\om_0}^{\om_1}}:\FH(M,\om_0)\to\FH(M,\om_1).
\eeq
In a similar way we can construct
\beqn
\widetilde{\Psi_{\om_1}^{\om_0}}:\FH(M,\om_1)\to\FH(M,\om_0),
\eeq
by following the homotopy backwards.
By a homotopy-of-homotopies argument, we conclude
$\widetilde{\Psi_{\om_1}^{\om_0}}\circ\widetilde{\Psi_{\om_0}^{\om_1}}=\id_{\FH(M,\om_0)}$ and
$\widetilde{\Psi^{\om_1}_{\om_0}}\circ\widetilde{\Psi^{\om_0}_{\om_1}}=\id_{\FH(M,\om_1)}$.
Therefore $\widetilde{\Psi_{\om_0}^{\om_1}}$ is an isomorphism with inverse $\widetilde{\Psi^{\om_0}_{\om_1}}$.
\end{proof}

\begin{Rmk}
In the proof of Theorem \ref{thm:continuation}, the {\em quadratic isoperimetric inequality} is essential.
One can check that if $u_{\om_1-\om_0}(t)\gnsim t$, the above proof does not work anymore.
\end{Rmk}


\section{Rabinowitz Floer homology}

\subsection{RFH for the cotangent bundle endowed with its canonical symplectic form}
In this section, we consider the cotangent bundle $(T^*N,\om_0=d\lambda_{\rm liou})$
of a closed Riemannian manifold $(N,g)$
where $\lambda_{\rm liou}=p\wedge dq$ is the {\em Liouville 1-form} for canonical coordinates $(q,p)\in T^*N$.
On the exact symplectic manifold $(T^*N,\lambda_{\rm liou})$, the {\em Liouville vector field} $X$
is defined by $\iota_X \om_0=\lambda_{\rm liou}$.
$(T^*N,\lambda_{\rm liou})$ is {\em complete and convex} i.e. the following conditions hold:
\begin{itemize}
\item There exists a compact subset $K\subset T^*N$ with smooth boundary such that $X$ points out of $K$ along $\p K$;
\item The vector field $X$ is complete and has no critical point outside $K$.
\end{itemize}
Equivalently, $(T^*N,\lambda_{\rm liou})$ is complete and convex
since there exists an embedding $\phi:\Sigma\times [1,\infty)\to
T^*N$ such that $\phi^*\lambda=r\alpha_{\Sigma}$, where $r$ denotes
the coordinates on $[1,\infty)$ and $\alpha_{\Sigma}$ is a contact
form, and such that $T^*N\setminus\phi(\Sigma\times(1,\infty))$ is
compact.

Consider now the complete convex exact symplectic manifold
$(T^*N,\lambda_{\rm liou})$ and the compact subset $DT^*N\subset
T^*N$ with smooth boundary $\Sigma:=ST^*N=\p DT^*N$ such that
$\lambda_{\rm liou}|_{ST^*N}$ is a positive contact form with a Reeb
vector field $R$. We abbreviate by
$\L:=\L_{T^*N}=C^{\infty}(S^1,T^*N)$ the free loop space of $T^*N$.
A {\em defining Hamiltonian} for $\Sigma$ is a smooth function
$H:T^*N\to\R$ with regular level set $\Sigma=H^{-1}(0)$ whose {\em
Hamiltonian vector field} $X_H$ has compact support and agrees with
$R$ along $\Sigma$. Given such a Hamiltonian, the {\em Rabinowitz
action functional} is defined by
\[\A_H : \L\times\R\to\R\]
\[\A_H(x,\eta):=\int_0^1x^*\lambda-\eta\int_0^1H(x(t))dt.\]
Critical points of $\A_H$ are solutions of the equations
\begin{equation}\label{eq:crit1}
\left.
\begin{array}{cc}
\p_t x(t)=\eta X_H(x(t)), & t \in \R/\Z \\
\int_0^1H(x(t))dt=0. &
\end{array}
\right\}
\end{equation}
By the first equation $H$ is constant along $x$, so the second equation implies $H(x(t))\equiv 0$.
Since $X_H=R$ along $\Sigma$, the equations (\ref{eq:crit1}) are equivalent to
\beqn\label{eq:crit2}
\left.
\begin{array}{cc}
\p_t x(t)=\eta R(x(t)), & t \in \R/\Z \\
x(t)\in\Sigma, & t \in \R/\Z.
\end{array}
\right\}
\eeq
So there are three types of critical points i.e. closed Reeb orbits on $\Sigma$:
\begin{itemize}
\item Positively parametrized closed Reeb orbits corresponding to $\eta>0$;
\item Negatively parametrized closed Reeb orbits corresponding to $\eta<0$;
\item Constant loops on $M$ corresponding to $\eta=0$.
\end{itemize}
The action of a critical point $(x,\eta)$ is $\A_H(s,\eta)=\eta$.

A compatible almost complex structure $J$ on part of the symplectization
$(\Sigma\times\R_+,d(r\alpha_{\Sigma}))$ of a contact manifold
$(\Sigma,\alpha_{\Sigma})$ is called {\em cylindrical} if it satisfies:
\begin{itemize}
\item $J$ maps the Liouville vector field $r\p_r$ to the Reeb vector field $R$;
\item $J$ preserves the contact distribution $\ker\alpha_{\Sigma}$;
\item $J$ is invariant under the Liouville flow $(y,r)\mapsto (y,e^tr),\;t\in\R.$
\end{itemize}

For a smooth family $(J_t)_{t\in S^1}$ of cylindrical \acs s on $(T^*N,\lambda_{\rm liou})$
we consider the following metric
$g=g_J$ on $\L\times\R$. Given a point $(x,\eta)\in\L\times\R$ and two tangent vectors
$(\hat{x}_1,\hat{\eta}_1),\;(\hat{x}_2,\hat{\eta}_2)\in T_{(x,\eta)}(\L\times\R)=\Gamma(S^1,x^*T(T^*N))\times\R$
the metric is given by
\[
g_{(x,\eta)}((\hat{x}_1,\hat{\eta}_1),(\hat{x}_2,\hat{\eta}_2))
=\int_0^1\om\left(\hat{x}_1,J_t(x(t))\hat{x}_2 \right)dt+\hat{\eta}_1\cdot\hat{\eta}_2.
\]
The gradient of the Rabinowitz action functional $\A_H$ \wrt the metric $g_J$ at a point $(x,\eta)\in \L\times\R$ reads
\beqn
\nabla\A_H(x,\eta)=\nabla_J\A_H(x,\eta)=
\begin{pmatrix}
-J_t(x)\left(\p_tx-\eta X_H(x) \right) \\
-\int_0^1H(x(t))dt
\end{pmatrix}.
\eeq
Hence the positive gradient flow lines are solutions $(x,\eta)\in C^{\infty}(\R\times S^1,T^*N)\times C^{\infty}(\R,\R)$
of the partial differential equation
\beqn\label{eq:rfhgrad}
\left.
\begin{array}{cc}
\p_sx+J_t(x)\left(\p_tx-\eta X_H(x) \right)=0 \\
\p_s\eta+\int_0^1H(x(t))dt=0
\end{array}
\right\}.
\eeq
Then for $-\infty<a<b\leq\infty$ the resulting truncated Floer homology groups
\[\RFHb^{(a,b)}(\Sigma,T^*N):=\HM^{(a,b)}(\A_H,J),\]
corresponding to action values in $(a,b)$, are well-defined and do not depend on the choice of the cylindrical $J$
and the defining Hamiltonian $H$. The {\em Rabinowitz Floer homology} of $(\Sigma,T^*N)$ is defined as the limit
\[\RFHb_*(\Sigma,T^*N):=
\lim_{\stackrel\longrightarrow a} \lim_{\stackrel \longleftarrow
b}\RFHb_*^{(-a,b)}(\Sigma,T^*N), \qquad a,b\to\infty.\] This
definition is equivalent to the original one in \cite{CF09} by
\cite[Theorem A]{CF09'}.

Since the Rabinowitz action functional is defined on the full loop space
and the first part of the differential in the Rabinowitz Floer complex counts topological cylinders,
we can split the Rabinowitz Floer homology into factors labeled by free homotopy classes
\[
\RFHb(\Sigma,T^*N)=\bigoplus_{\nu\in[S^1,T^*N]}\RFH^\nu(\Sigma,T^*N),
\]
where $\RFH^\nu(\Sigma,T^*N)$ is the Rabinowitz Floer homology for the Rabinowitz action functional
restricted to $\L^\nu=\L^\nu_{T^*N}$.

\subsubsection{Index and grading convention}
Let $\M$ be the moduli space of all finite energy gradient flow lines of the action functional
$\A_H:\L\times\R\to\R$. Since $\A_H$ is Morse-Bott, every finite energy gradient flow line
$(v,\eta)\in C^\infty(\R\times S^1,V)\times C^\infty(\R,\R)$ converges exponentially at both
ends to critical points $(v_\pm,\eta_\pm)\in\Crit(\A_H)$ as the flow parameter goes to $\pm\infty$.
The linearization of the gradient flow equation along any path $(v,\eta)$ in $\L\times\R$ which converges
exponentially to the critical point of $\A_H$ gives rise to a Fredholm operator $D^{\A_H}_{(v,\eta)}$.
Let $C^-,C^+\subset\Crit(\A_H)$ be the connected component of the critical manifold of $\A_H$
containing $(v_-,\eta_-),(v_+,\eta_+)$ respectively. The local virtual dimension of $\M$ at a finite energy
gradient flow line is defined to be
\[
\mathrm{virdim}_{(v,\eta)}\M:=\ind D^{\A_H}_{(v,\eta)}+\dim C^-+\dim C^+
\]
where $\ind D^{\A_H}_{(v,\eta)}$ is the Fredholm index of the Fredholm operator $\ind D^{\A_H}_{(v,\eta)}$.
For generic compatible almost complex structures, the moduli space of finite energy gradient flow lines is
a manifold and the local virtual dimension of the moduli space at a gradient flow line $(v,\eta)$ corresponds
to the dimension of the connected component of $\M$ containing $(v,\eta)$.

To define a $\Z$-grading on $\RFHb(\Sigma,T^*N)$, we need that the
local virtual dimension just depends on the asymptotics of the
finite energy gradient flow line. Since the first Chern class of
$T^*N$ vanishes, it can be shown that the local virtual dimension
equals
\[
\mathrm{virdim}_{(v,\eta)}\M=\CZ(v_+)-\CZ(v_-)+\frac{\dim C^-+\dim C^+}{2}.
\]
In order to deal with the third term it is useful to introduce the following index for the Morse function $h$
on $\Crit(\A_H)$. We define the {\em signature index} $\ind^\sigma_h(c)$ of a critical point $c$ of $h$ to be
\[
\ind^\sigma_h(c):=-\frac{1}{2}\mathrm{sign}(\Hess_h(c)).
\]
We define a {\em grading} $\mu$ on $\RFCb(\Sigma,T^*N)=\CM(\A_H,h)$ by
\[
\mu(c):=\CZ(c)+\ind^\sigma_h(c)+\frac{1}{2}.
\]
These define a $\Z$-grading on the homology $\RFHb(\Sigma,T^*N)$. We
refer to \cite{CF09, MerPat10} for more details.

\subsection{Rabinowitz Floer homology for a twisted cotangent bundle}
In the previous section, we considered an exact symplectic manifold.
By the exactness of symplectic form, there is no need to care about the filling disk of a given loop.
In general, twisted symplectic forms are not exact.
In order to define Rabinowitz Floer homology,
we need the notions of a {\em symplectically atoroidal manifold}
and a {\em virtual restricted contact type hypersurface}.

\begin{Def}
A symplectic manifold $(M,\om)$ is called {\em symplectically atoroidal} if
\[
\int_{\T^2}f^*\om=0,
\]
for any smooth function $f:\T^2\to T^*N$.
\end{Def}

\begin{Rmk}
Since there is a map $g:\T^2\to S^2$ of non vanishing degree, {\em
symplectically atoroidal} implies {\em symplectically aspherical}.
\end{Rmk}

\begin{Lemma}[Merry \cite{Mer10}]\label{lem:ato}
Let $\sigma\in\Om^2(N)$ be a weakly exact 2-form and $u_{\sigma}\sim 1$,
then $f^*\sigma$ is exact for any smooth map $f:\T^2\to N$.
\end{Lemma}
\begin{proof}
Consider $G:=f_*(\pi_1(\T^2))\leq\pi_1(N)$.
Then $G$ is amenable, since $\pi_1(\T^2)=\Z^2$, which is amenable.
Now Lemma 5.3 in \cite{Pat06} tells us that since $\|\theta\|_\infty<\infty$,
we can replace $\theta$ by a $G$-invariant primitive $\theta'$ of $\widetilde\sigma$,
which descends to a primitive $\theta''\in\Om^1(\T^2)$ of $f^*\sigma$.
\end{proof}

\begin{Rmk}\label{rmk:fhc}
Given a free homotopy class $\nu\in[S^1,T^*N]$, fix a reference loop
$v_{\nu}=(q_\nu,p_\nu)\in\L^\nu_{T^*N}$. Let $Z$ be a cylinder
$S^1\times[0,1]$ with two boundary components $\p'Z$ with the
boundary orientation and $\p''Z$ with the opposite boundary
orientation. Choose $\overline v:Z\to T^*N$ any smooth map such that
$\overline v|_{\p'Z}=v$ and $\overline v|_{\p''Z}=v_\nu$. Then
thanks to the previous lemma, the integral $\int_Z\overline
v^*\tau^*\sigma$ is independent of the choice of $\overline v$.
Similarly given any $q\in\L^{\tau\nu}_N$, let $\overline q:Z\to N$
denote any smooth map such that $\overline q|_{\p'Z}=q$ and
$\overline q|_{\p''Z}=q_\nu$. Then the integral $\int_Z\overline
q^*\sigma$ is independent of the choice of $\overline q$. Note that
in particular if $q=\tau\circ v$ then
\[
\int_Z\overline v^*\tau^*\sigma=\int_Z\overline q^*\sigma.
\]
In particular, let $\sigma\in\Om^2(N)$ be a weakly exact 2-form satisfying $u_\sigma\sim 1$,
then the twisted cotangent bundle $(T^*N,\om_\sigma)$ is a {\em symplectically atoroidal manifold}.
Moreover, the Rabinowitz action functional
\[
\A_{\om_\sigma}:\L\times\R\to\R
\]
\[
\A_{\om_\sigma}(v,\eta):=\A_{H,\om_\sigma}(v,\eta)=\int_Z\overline
v^*\om_\sigma-\eta\int_0^1H(v(t))dt
\]
is well-defined, independent of the choice of $\overline v$. In the
special case, where $\nu=0$ is the trivial free homotopy class we
choose $v_\nu$ just a constant loop. In this case the cylinder $Z$
can be replaced by a filling disk $\D^2$ for the loop $v$.
\end{Rmk}

\begin{Def}
A closed hypersurface $\Sigma$ in a symplectic manifold $(M,\om)$ is called {\em virtually contact},
if there is a covering $p:\widehat M\to M$ and a primitive $\lambda\in\Om^1(\widehat\Sigma)$ of $p^*\om$ such that
\bea\label{eqn:vcp}
\sup_{x\in\widehat{\Sigma}}|\lambda_x|\leq C<\infty,\qquad
\inf_{x\in\widehat{\Sigma}}\lambda(R)\geq\mu>0,
\eea
where $|\cdot|$ is the lifting of a metric on $\Sigma$ and $R$ is the pullback of a vector field generating $\ker(\om|_{\Sigma})$.
\end{Def}

\begin{Def}
A closed hypersurface $\Sigma$ in a symplectic manifold $(M,\om)$ is called
{\em virtual restricted contact}, if there is a covering $p:\widehat M\to M$
and a primitive $\lambda\in\Om^1(\widehat M)$ of $p^*\om$ such that
$\lambda$ satisfy (\ref{eqn:vcp}) again on $\widehat\Sigma$.
\end{Def}

\begin{Rmk}
A {\em virtual restricted contact homotopy} is a smooth homotopy
$(\Sigma_t,\lambda_t)\subset(M,\om)$ of virtual restricted contact
hypersurfaces with the corresponding 1-forms on the covers such that
the preceding conditions hold with constants $C,\mu$ independent of
$t$. $\RFH(\Sigma,M)$ is defined for each virtual restricted contact
hypersurface $\Sigma$ and is invariant under virtual restricted
contact homotopies. For a twisted cotangent bundle
$(T^*N,\om_\sigma)$ with any $k\in\R$ above Ma\~n\'e critical value
$c=c(g,\sigma,U)$ the hypersurface $\Sigma_k=H^{-1}_U(k)\subset
T^*N$ is virtual restricted contact, see \cite{CFP09}.
\end{Rmk}

\section{Continuation homomorphism in RFH for symplectic deformations}
Let us begin with the {\em defining Hamiltonian} $H$ of the virtual restricted contact hypersurface $\Sigma_k\subset T^*N$
\beqn
H:=H_{U,k,\xi}=\beta_{\xi}\circ (H_U-k)
\eeq
where, $\beta_{\xi}(t)$ is a smooth cut-off function satisfying $0\leq\dot{\beta}_{\xi}\leq1$,
\beqn
\beta_{\xi}(t)=
\left\{
\begin{array}{rcr}
t&\text{ if }&|t|\leq\xi-\epsilon \\
\xi&\text{ if }&t\geq\xi+\epsilon \\
-\xi&\text{ if }&-t\geq\xi+\epsilon
\end{array}
\right.,
\qquad \epsilon=\min\left\{1/3,\xi/3\right\}.
\eeq
Now we define the Rabinowitz action functional given by
\[\A_{\om_\sigma}:\L\times\R\to\R\]
\beqn\label{eq:rfhaf}
\A_{\om_\sigma}(v,\eta):=\A_{H,\om_{\sigma}}(v,\eta)=\int_{Z}\overline{v}^*\om_{\sigma}-\eta\int_0^1H(v(t))dt,
\eeq
where $\overline{v}, Z$ are given in Remark \ref{rmk:fhc}.

In this section, we consider the canonical cotangent bundle $(T^*N,\om_0)$
and the twisted cotangent bundle $(T^*N,\om_{\sigma})$
with the virtual restricted contact hypersurface
$\Sigma_k=H^{-1}(0)=H^{-1}_U(k)$ where $k>c(g,\sigma,U)$
and $H_U(q,p)=\frac{1}{2}|p|_g^2+U(q)$.
For convenience, let us define the following sets
\bean
&\Mp(N)=\{\sigma\in\Om^2(N)\ |\ \widetilde\sigma=d\theta,\ \|\theta\|_{\infty}<\infty\}; \\
&\Om^{\Mp}(T^*N)=\{\om_{\sigma}\in\Om^{2}(T^*N)\ |\ \sigma\in\Mp(N)\}; \\
&\Om^{\Mp}(\Sigma_k)=\{\om_{\sigma}\in\Om^{\Mp}(T^*N)\ |\ k>c(g,\sigma,U)\}; \\
&\Om^{\Mp}_{\rm reg}(\Sigma_k)=\{\om_{\sigma}\in\Om^{\Mp}(\Sigma_k)\
|\ \A_{\om_\sigma}:\L\times\R\to\R\text{ is Morse-Bott}\}. \eea Note
that $\Om^{\Mp}(\Sigma_k)$ is convex. Indeed this follows from the
following estimate for $t \in [0,1]$ and for primitives $\theta_1$
and $\theta_2$
\begin{eqnarray*}
|t\theta_1+(1-t)\theta_2|^2 &\leq&
t^2|\theta_1|^2+2t(1-t)|\theta_1||\theta_2|+(1-t)^2|\theta_2|^2\\
&\leq&t^2|\theta_1|^2+t(1-t)|\theta_1|^2+t(1-t)|\theta_2|^2+(1-t)^2
|\theta_2|^2\\
&=&t|\theta_1|^2+(1-t)|\theta_2|^2.
\end{eqnarray*}
It is known that for surfaces and vanishing potential the set
$\Om^{\Mp}_{\rm reg}(\Sigma_k)$ is dense in $\Om^{\Mp}(\Sigma_k)$
thanks to work of Miranda, see \cite{Mir}. In higher dimensions this
seems to be an open problem, although it would be very surprising if
it failed.


For a pair $(\om_0,\om_{\sigma})$ of $\Om_{\rm
reg}^{\Mp}(\Sigma_k)$, we construct the continuation homomorphism
\[
\widetilde{\Psi_{\om_0}^{\om_{\sigma}}}_*:\RFHb_*(\Sigma_k,\om_0)\to
\RFHb_*(\Sigma_k,\om_{\sigma}),
\]
by counting solutions of an $s$-dependent Rabinowitz Floer equation.
Before the construction, we must check the $L_{\infty}$-bound of the
Lagrange multiplier $\eta$ in the case of a twisted cotangent bundle with virtual restricted contact hypersurface.
The proof of the following proposition proceeds as \cite{CF09} for the restricted contact type case.
It was already used with no explicit proof in \cite{CFP09}.
For the readers convenience we include a proof here.

\begin{Prop}\label{prop:lagbd}
Let $(T^*N,\om_{\sigma})$ be a twisted cotangent bundle
with a virtual restricted contact hypersurface $\Sigma_k=H^{-1}(0)$ where $k>c$.
Then there exist constants $\epsilon>0$ and $\overline{c}<\infty$ such that the following holds
\beq\label{eqn:lagbd}
\|\nabla\A_{\om_\sigma}(v,\eta)\|\leq\epsilon \Longrightarrow |\eta|\leq \overline{c}(|\A_{\om_\sigma}(v,\eta)|+1).
\eeq
\end{Prop}
Before embarking on the proof of the Proposition \ref{prop:lagbd}, we first explain as a
warm-up an extremal case of it, namely instead of looking at almost
critical points we consider critical points themselves.

\begin{Lemma}
Under the same assumptions as in Proposition \ref{prop:lagbd} for $(v,\eta)\in\Crit(\A_{\om_\sigma}|_{\L^0\times\R})$ we have
\[|\A_{\om_\sigma}(v,\eta)|\geq\frac{|\eta|}{c'}\]
where $c'>0$.
\end{Lemma}

\begin{proof}
Since we are considering the component of contractible loops we only
have to choose a filling disc $\overline{v}\colon \D^2 \to T^*N$ for
the loop $v$. Inserting (\ref{eq:crit1}) into $\A_{\om_\sigma}$ and
using the assumption of virtual restricted contact type
\bea\label{eqn:aas43}
|\A_{\om_\sigma}(v,\eta)|&=\left|\int_{\D^2}\overline{v}^*\om_{\sigma}\right|
=\left|\int_{\D^2}\tilde{\overline{v}}^*\pi^*\om_{\sigma}\right|
=\left|\int_{\D^2}\tilde{\overline{v}}^*d\lambda_{\sigma}\right|
=\left|\int_{S^1}\tilde{v}^*\lambda_{\sigma}\right| \\
&=\left|\int_0^1\lambda_{\sigma}(\partial_t\tilde{v})\right|
=\left|\int_0^1\lambda_{\sigma}(\eta\widetilde{X_H^{\om_{\sigma}}}(\tilde{v}))\right| \\
&=\left|\eta\int_0^1\lambda_{\sigma}(\widetilde{X_H^{\om_{\sigma}}}(\tilde{v}))\right| \\
&\geq\frac{|\eta|}{c'}
\eea
where $\tilde{\overline{v}},\;\widetilde{X_H^{\om_{\sigma}}}$ are lifts of $\overline{v},\;X_H^{\om_{\sigma}}$ respectively
to the cover $\pi:\widetilde{\Sigma}\to\Sigma$.
The constant $c'>0$ exists by the second inequality in (\ref{eqn:vcp}).
\end{proof}

\begin{proof}[{\bf Proof of Proposition \ref{prop:lagbd}}]
The proof consists of 4 steps. \vspace{3mm}
\\
{\bf Step 1} : {\em There exist $\delta>0$, a constant $c'>0$, a
covering $p \colon \widehat{M} \to M=T^*N$ and a primitive $\lambda
\in \Omega^1(\widehat{M})$ of $p^*\omega_\sigma$ such that on
$U_\delta =H^{-1}(-\delta,\delta)$ we have the estimates
$\lambda(\widetilde{X_H^{\om_{\sigma}}})>\frac{1}{2c'}+\delta$ and
$||\lambda||_\infty<\infty$ where the $L^\infty$-norm is taken with
respect to the lift of a metric on $N$.}
\\ \\
The assertion of Step~1 is surely true for $\delta=0$ for a
primitive $\lambda_0 \in \Omega^1(\widehat{M})$ just by definition
of virtual restricted contact type. Let
$\widehat{\Sigma}_k=p^{-1}(\Sigma_k)$ be the lift of $\Sigma_k$.
Thank to the bounds in the virtual contact assumption we can find
$\epsilon>0$ and a diffeomorphism $\Psi$ from
$\widehat{\Sigma}_k\times (-\epsilon,\epsilon)$ to an open
neighbourhood $U$ of $\widehat{\Sigma}_k$ in $\widehat{M}$ such that
$\Psi$ pulls back the symplectic form $\omega_\sigma$ on $U$ to the
symplectic form $\omega=d(r\lambda_0|_{\widehat{\Sigma}_k})$ on
$\widehat{\Sigma}_k\times (-\epsilon,\epsilon)$. Now choose
$\delta>0$ so small that $U_\delta \subset
\Psi\big(\widehat{\Sigma}_k\times (-\epsilon/2,\epsilon/2)\big)$ and
the bounds required in Step~1 hold for $\Psi_*\lambda_1$ with
$\lambda_1=r\lambda_0|_{\widehat{\Sigma}_k}$. Since $\lambda_1$ and
$\Psi^*\lambda_0$ are two primitives of $\Psi^*\omega_\sigma$ which
coincide on $\widehat{\Sigma}_k$ we conclude that there exists a
function $f \in
C^\infty\big(\widehat{\Sigma}_k\times(-\epsilon,\epsilon)\big)$ such
that $\lambda_1=\Psi^*\lambda_0+df$. Now choose a cutoff function
$\beta \in C^\infty\big(\widehat{\Sigma}_k\times
(-\epsilon,\epsilon) \big)$ with the property that $\beta(x,r)=1$
for $x \in \widehat{\Sigma}_k$ and $|r| \leq \epsilon/2$ and
$\beta(x,r)=0$ if $r \geq 3\epsilon/4$. Finally set
$\lambda=\lambda_0 +\Psi_*d(\beta f)$ on
$\Psi\big(\widehat{\Sigma}_k\times (-\epsilon,\epsilon)\big)$. This
finishes the proof of Step~1. For the next step we fix $\nu\in
[S^1,T^*N]$.
\\ \\
{\bf Step 2} : {\em There exist $\delta>0$ and a constants
$c_{\delta}<\infty$ and $a^\nu \in \mathbb{R}$ with the following
property. For every $(v,\eta)\in\L^\nu\times\R$ such that $v(t)\in
U_{\delta}=H^{-1}(-\delta,\delta)$ for every $t\in\R/\Z$, the
following estimate holds:}
\[
|\eta|\leq2c'|\A_{\om_\sigma}(v,\eta)|+c_{\delta}\|\nabla\A_{\om_\sigma}(v,\eta)\|+2c'|a^\nu|.
\]
Let $c'$, $\delta$ and $\lambda$ be as in Step~1 and set
\beq\label{eqn:c_deltaes}
c_{\delta}=2c'\|\lambda|_{\pi^{-1}(U_{\delta})}\|_{\infty}<\infty.
\eeq We estimate \bean
|\A_{\om_\sigma}(v,\eta)|&=\left|\int_0^1\lambda(\tilde{v})(\partial_t\tilde{v})
-\underbrace{\int_0^1\lambda(\tilde{v}_\nu)(\partial_t\tilde{v}_\nu)}_{:=a^\nu}
-\eta\int_0^1H(v(t))dt \right| \\
&=\left|\eta\int_0^1\lambda(\tilde{v})(\widetilde{X_H^{\om_{\sigma}}}(\tilde{v}))
+\int_0^1\lambda(\tilde{v})(\partial_t\tilde{v}-\eta\widetilde{X_H^{\om_{\sigma}}}(\tilde{v}))
-a^\nu
-\eta\int_0^1H(v(t))dt \right| \\
&\geq\left|\eta\int_0^1\lambda(\tilde{v})(\widetilde{X_H^{\om_{\sigma}}}(\tilde{v}))\right|
-\left|\int_0^1\lambda(\tilde{v})(\partial_t\tilde{v}-\eta\widetilde{X_H^{\om_{\sigma}}}(\tilde{v}))\right|
-\left|\eta\int_0^1H(v(t))dt \right|
-|a^\nu|\\
&\geq|\eta|\left(\frac{1}{2c'}+\delta\right)-\frac{c_{\delta}}{2c'}\|\partial_t\tilde{v}-\eta\widetilde{X_H^{\om_{\sigma}}}(\tilde{v})\|_1-|\eta|\delta-|a^\nu| \\
&\geq\frac{|\eta|}{2c'}-\frac{c_{\delta}}{2c'}\|\partial_t\tilde{v}-\eta\widetilde{X_H^{\om_{\sigma}}}(\tilde{v})\|_2-|a^\nu| \\
&\geq\frac{|\eta|}{2c'}-\frac{c_{\delta}}{2c'}\|\nabla\A_{\om_\sigma}(v,\eta)\|-|a^\nu|,
\eea where $v_\nu\in\L^\nu$ is a reference loop defined in Remark
\ref{rmk:fhc}. This proves Step 2.
\\ \\
{\bf Step 3} : {\em For each $\delta>0$, there exists
$\epsilon=\epsilon(\delta)>0$ such that if
$\|\nabla\A_{\om_\sigma}(v,\eta)\|\leq\epsilon$ then $v(t)\in
U_{\delta}$ for every $t\in[0,1]$.}
\\ \\
First assume that $v\in\L$ has the property that there exist $t_0, t_1\in\R/\Z$ such that
$|H(v(t_0))|\geq\delta$ and $|H(v(t_1))|\leq\delta/2$. We claim that
\beq\label{2-1}
\|\nabla\A_{\om_\sigma}(v,\eta)\|\geq\frac{\delta}{2\kappa}
\eeq
for every $\eta\in\R$, where
\[
\kappa:=\max_{x\in\overline{U}_{\delta}}\|\nabla H(x)\|.
\]
To see this, assume without loss of generality that $t_0<t_1$ and
$\delta/2\leq|H(v(t))|\leq\delta$ for all $t\in[t_0,t_1]$. Then we estimate
\bean
\|\nabla\A_{\om_\sigma}(v,\eta)\|&\geq \sqrt{\int_0^1\|\p_tv-\eta X_H^{\om_{\sigma}}(v)\|^2dt} \\
&\geq \int_0^1\|\p_tv-\eta X_H^{\om_{\sigma}}(v)\|dt \\
&\geq \int_{t_0}^{t_1}\|\p_tv-\eta X_H^{\om_{\sigma}}(v)\|dt \\
&\geq \frac{1}{\kappa}\int_{t_0}^{t_1}\|\nabla H(v)\|\cdot\|\p_tv-\eta X_H^{\om_{\sigma}}(v)\|dt \\
&\geq \frac{1}{\kappa}\int_{t_0}^{t_1}|\langle\nabla H(v),\p_tv-\eta X_H^{\om_{\sigma}}(v)\rangle|dt \\
&= \frac{1}{\kappa}\int_{t_0}^{t_1}|\langle\nabla H(v),\p_tv\rangle|dt \\
&= \frac{1}{\kappa}\int_{t_0}^{t_1}|dH(v)\p_tv|dt \\
&= \frac{1}{\kappa}\int_{t_0}^{t_1}|\p_tH(v)|dt \\
&\geq \frac{1}{\kappa}\left|\int_{t_0}^{t_1}\p_tH(v)dt\right| \\
&= \frac{1}{\kappa}|H(v(t_1))-H(v(t_0))| \\
&\geq \frac{1}{\kappa}\left(|H(v(t_1))|-|H(v(t_0))|\right) \\
&\geq \frac{\delta}{2\kappa}. \eea Now assume that $v\in\L$ has the
property that $v(t)\in T^*N\setminus U_{\delta/2}$ for every
$t\in[0,1]$. In this case we estimate \beq\label{2-2}
\|\nabla\A_{\om_\sigma}(v,\eta)\|\geq\left|\int_0^1H(v(t))dt\right|\geq\frac{\delta}{2}
\eeq for every $\eta\in\R$. From (\ref{2-1}) and (\ref{2-2}) Step 3
follows with
\[
\epsilon=\frac{\delta}{2\max\{1,\kappa\}}.
\]
\\ \\
{\bf Step 4} : {\em We prove the proposition.}
\\ \\
Choose $\delta$ as in Step 1, $\epsilon=\epsilon(\delta)$ as in Step
3 and
\[
\overline{c}=\max\{2c',2c_{\delta}\epsilon,4c'|a^\nu|\}.
\]
Assume that $\|\nabla\A_{\om_\sigma}(v,\eta)\|\leq\epsilon$ then
\[
|\eta|\leq2c'|\A_{\om_\sigma}(v,\eta)|+c_{\delta}\|\nabla\A_{\om_\sigma}(v,\eta)\|+2c'|a^\nu|
\leq \overline{c}(|\A_{\om_\sigma}(v,\eta)|+1).
\]
This proves the Proposition \ref{prop:lagbd}.
\end{proof}

\begin{Rmk}\label{rmk:constconti}
A careful inspection of the proof of Proposition \ref{prop:lagbd}
shows that the constant $c',\delta,c_{\delta},\epsilon(\delta),$ and
$\overline{c}$ continuously depend on the 2-form
$\sigma\in\Om^2(M)$. In particular, Proposition \ref{prop:lagbd} can
be extended to families of symplectic forms.
\end{Rmk}

\begin{Lemma}[Linear isoperimetric inequality]\label{lem:linhmtp}
Let $\sigma\in\Om^2(N)$ be a weakly exact 2-form and $u_{\sigma}\sim 1$, then
\[
\int_Z\overline q^*\sigma \leq C\left(\int_0^1|\p_tq|dt+1\right)
\]
where, $\overline q,Z$ are the same as in Remark \ref{rmk:fhc} and
$C=C(N,g,\sigma,q_\nu)$.
\end{Lemma}
\begin{proof}
The proof uses the same argument as in Lemma
\ref{thm:isoperimetric}. Let $\widetilde{\overline q}:Z\to
\widetilde N$ be the lifting of $\overline q$ to the universal cover
and $\theta\in\Om^1(\widetilde N)$ be a bounded primitive of
$\widetilde\sigma$ as in Lemma \ref{lem:ato}. Then we get \bean
\int_Z\overline q^*\sigma&=\int_Z\widetilde{\overline q}^*\widetilde\sigma \\
&=\int_Z\widetilde{\overline q}^*d\theta \\
&=\int_R\widetilde q^*\theta \\
&\leq\left|\int_0^1\widetilde q^*\theta\right|+\left|\int_0^1\widetilde q_\nu^*\theta\right|
+\left|\int_0^1\underline r^*\theta\right|+\left|\int_0^1\overline r^*\theta\right| \\
&\leq\|\theta\|_\infty\left(\int_0^1|\p_t q|dt+\int_0^1|\p_t
q_\nu|dt +\int_0^1|\p_t\underline r|dt+\int_0^1|\p_t\overline
r|dt\right), \eea where $R$ is a rectangle in $\widetilde N$ which
consists of $\widetilde q,\;\widetilde q_\nu,\;\underline r$, and
$\overline r$. Here, $\underline r:[0,1]\to\widetilde N$ is a path
from $\widetilde q(0)$ to $\widetilde q_\nu(0)$ and $\overline
r:[0,1]\to\widetilde N$ is a path from $\widetilde q(1)$ to
$\widetilde q_\nu(1)$. Since $\int_Z\overline q^*\sigma$ does not
depend on the choice of $Z$, we may assume that $\underline r$,
$\overline r$ are length minimizing curves on $\widetilde N$. This
implies that $\underline r$, $\overline r$ are geodesics contained
in a fundamental domain in $\widetilde N$ or
\[
\int_0^1|\p_t\underline r|dt\leq\text{diam}(N),\quad\int_0^1|\p_t\overline r|dt\leq\text{diam}(N).
\]
Set
$C=\max\left\{\|\theta\|_\infty,2\|\theta\|_\infty\int_0^1|\p_tq_\nu|dt,4\|\theta\|_\infty
\text{diam}(N)\right\}$ then we get the conclusion.
\end{proof}


\begin{Rmk}\label{rmk:conti_const}
Note that $C$ converges to $0$ as $|\sigma|_g\to0$.
\end{Rmk}

\begin{Rmk}\label{rmk:lip_pro}
If we consider the family of symplectic forms on $T^*N$
\[
\om_s=\om_0+\beta(s)\tau^*\sigma\in\Om^{\Mp}(\Sigma_k),\qquad
\forall s\in\R,
\]
where $\beta(s)\in C^\infty(\R,[0,1])$ is a cut-off function
satisfying $\beta(s)=1$ for $s\geq1$, $\beta(s)=0$ for $s\leq0$, and
$0\leq\dot{\beta}(s)\leq2$, then we obtain the estimate \bean \left|
\int_Z\overline{v}^*\dot\om_s \right|
&\leq \left| \int_Z\overline{v}^*\dot\beta(s)\tau^*\sigma\right| \\
&=\dot\beta(s)\left| \int_Z\overline{v}^*\tau^*\sigma \right| \\
&\leq C\dot\beta(s)\left(\int_{S^1} |\partial_t v(t)| dt+1\right),
\eea
for some $C=C(N,g,\sigma,q_\nu)$ given in Lemma \ref{lem:linhmtp}.
\end{Rmk}

\begin{Prop}\label{prop:lagbd2}
Let $w=(v,\eta)\in C^{\infty}(\R\times S^1,T^*N)\times C^\infty(\R,\R)$ be a gradient flow line of
\[
\A_{\om(s)}(v,\eta):=\A_{H,\om_{s}}(v,\eta)=\int_{Z}\overline{v}^*\om_s-\eta\int_0^1H(x(t))dt
\]
i.e. a solution of
\beq\label{eqn:rfhgrd}
\left.
\begin{array}{cc}
\p_sv+J_{t,s}(v)\left(\p_tv-\eta X_H^{\om_s}(v) \right)=0 \\
\p_s\eta+\int_0^1H(v(t))dt=0
\end{array}
\right\}
\eeq

\beqn
\label{eqn:rfhlim}
\lim_{s \to -\infty}w(s)=w_-\in\Crit\A_{\om(0)}, \qquad
\lim_{s \to \infty}w(s)=w_+\in\Crit \A_{\om(1)},
\eeq
where $\om_s$ is same as in Remark \ref{rmk:lip_pro}.
If $|\sigma|_g$ is sufficiently small then the $L^{\infty}$-norm of $\eta$ is uniformly bounded
in terms of a constant which only depends on $w_-,w_+$.
\end{Prop}

\begin{proof}
We prove the proposition in three steps.

\vspace{3mm}

{\bf Step1} : Let us first {\em bound the energy of $w$ in terms of $\|\eta\|_{\infty}$.}
\bea\label{e0}
E(w)&=\intinf\|\p_sw\|_{s}^2 ds \\
&=\intinf\<\p_s w, \nabla \A_{\om(s)}(w) \rangle_{s}ds\\
&=\intinf\frac{d}{ds}\A_{\om(s)}(w)ds-\intinf\dot{\A}_{\om(s)}(w)ds \\
&=\A_{\om(1)}(w_+)-\A_{\om(0)}(w_-)-\intinf\dot{\A}_{\om(s)}(w)ds.
\eea
We estimate the third term by
\bea\label{e1}
\left| \intinf\dot{\A}_{\om(s)}(w)ds \right|
&\leq \intinf\left|\dot{\A}_{\om(s)}(w)\right|ds \\
&= \intinf\dot{\beta}(s) \left| \int_Z\overline{v}^*\tau^*\sigma \right| ds\\
&\leq \intinf \dot{\beta}(s) C\left( \int_{S^1}|\partial_{t}v|_{t,s} dt+1 \right) ds,
\eea
where $C$ is the isoperimetric constant in Remark \ref{rmk:lip_pro}
and $|\cdot|_{t,s}$ is the norm on $T^*N$ induced by the Riemannian metric $\om_s(\cdot,J_{t,s}\cdot).$
From the gradient flow equation (\ref{eqn:rfhgrd}) we get
\[
\p_tv=J_{t,s}(v)\p_sv+\eta X_H^{\om_s}(v).
\]
By putting this into (\ref{e1}), we then obtain
\bea\label{eqn:dotAbd}
\intinf\left|\dot{\A}_{\om(s)}(w)\right|ds
&\leq\intinf \dot{\beta}(s) C\left( \int_{S^1}|\partial_{t}v|_{t,s} dt +1\right) ds \\
&\leq\intinf \dot{\beta}(s) C\left( \int_{S^1}|J_{t,s}(v)\p_sv+\eta X_H^{\om_s}(v)|_{t,s} dt +1\right) ds \\
&\leq\intinf \underbrace{\dot{\beta}(s)}_{\leq 2} C\left( \int_{S^1}\left(|\p_sv|_{t,s}+|\eta|\; |X_H^{\om_s}(v)|_{t,s}\right) dt+1 \right) ds \\
&\leq2C\int_0^1 \left( \int_{S^1}\left(|\p_sv|_{t,s}^2+1+|\eta|\; |X_H^{\om_s}(v)|_{t,s}\right) dt+1 \right) ds \\
&\leq2CE(v)+4C+2C\|\eta\|_{\infty}c'' \\
&\leq2CE(w)+4C+2C\|\eta\|_{\infty}c'', \eea where
$c''=\max_{s\in[0,1]\atop v\in T^*N}|X_H^{\om_s}(v)|_{t,s}$. Note
that the maximum is attained, since by the assumption $dH$ has
compact support. Now by substituting the above equation into
(\ref{e0}), we get
\bea\label{eqn:e_bd}
E(w)&=\A_{\om(1)}(w_+)-\A_{\om(0)}(w_-)-\intinf\dot{\A}_{\om(s)}(w)ds \\
&\leq\A_{\om(1)}(w_+)-\A_{\om(0)}(w_-)+2CE(w)+4C+2C\|\eta\|_{\infty}c''
\eea
By choosing $\sigma\in\Om^2(M)$ with sufficiently small norm,
we may assume that the isoperimetric constant $C$ is less than $\frac{1}{4}$.
For simplicity, set $\Delta=\A_{\om(1)}(w_+)-\A_{\om(0)}(w_-)$, then we get
\bea\label{eqn:e_bd2}
E(w)&\leq2\A_{\om(1)}(w_+)-2\A_{\om(0)}(w_-)+8C+4C\|\eta\|_{\infty}c''\\
&=2\Delta+8C+4C\|\eta\|_{\infty}c''.
\eea
This finishes Step1.

\vspace{3mm}

{\bf Step2} : Let $\epsilon$ be as in Proposition \ref{prop:lagbd} and Remark \ref{rmk:constconti}.
For $l\in\R$ let $\tau(l)\geq0$ be defined by
\[
\tau(l):=\inf\{\tau\geq0:\|\nabla\A_{\om(s)}((v,\eta)(l+\tau))\|_s<\epsilon\}.
\]
In this step we {\em bound $\tau(l)$ in terms of $\|\eta\|_{\infty}$ for all $l\in\R$}. Namely
\bean
E(w)&=\intinf\|\p_sw\|_s^2ds\\
&=\intinf\|\nabla\A_{\om(s)}\|_s^2ds \\
&\geq\int_{l}^{l+\tau(l)}\underbrace{\|\nabla\A_{\om(s)}\|_s^2}_{\geq \epsilon^2}ds \\
&\geq\epsilon^2\tau(l)
\eea
Step1 and the above estimate finish Step2.

\vspace{3mm}

{\bf Step3} : {\em We prove the proposition.}\\
First set
\bean
\|H\|_{\infty}=\max_{x\in T^*N}|H(x)|, \quad
K=\max\{-\A_{\om(0)}(w_-),\A_{\om(1)}(w_+)\}.
\eea
By definition of $\tau(l)$, we obtain $\|\nabla\A_{\om(s)}((v,\eta)(l+\tau(l)))\|_s<\epsilon$.
Now we are able to use Proposition \ref{prop:lagbd} and
get the following estimate by using (\ref{eqn:lagbd}), (\ref{eqn:dotAbd}) and (\ref{eqn:e_bd})
\bea\label{eqn:eta_bd1}
|\eta(l+\tau(l))|
&\leq \overline{c}(|\A_{\om(s)}(w(l+\tau(l)))|+1) \\
&\leq \overline{c}\left(K+\intinf\left|\dot{\A}_{\om(s)}\right|ds+1\right) \\
&\leq \overline{c}\left(K+2CE(w)+4C+2C\|\eta\|_{\infty}c''+1\right) \\
&\leq \overline{c}\left(K+4C\Delta+16C^2+8C^2\|\eta\|_\infty c''+4C+2C\|\eta\|_{\infty}c''+1\right).
\eea
By Step2 and (\ref{eqn:e_bd}), we obtain the following inequalities
\bea\label{eqn:eta_bd2}
\left|\int_{l}^{l+\tau(l)}\dot{\eta}(s)ds\right|
&\leq \left|\int_{l}^{l+\tau(l)}\int_0^1H(v(t))dt\ ds\right|\\
&\leq \|H\|_{\infty}\tau(l) \\
&\leq \|H\|_{\infty}\frac{E(w)}{\epsilon^2} \\
&\leq \frac{\|H\|_{\infty}}{\epsilon^2}(2\Delta+8C+4C\|\eta\|_{\infty}c'').
\eea
Combining the above two estimates (\ref{eqn:eta_bd1}), (\ref{eqn:eta_bd2}),
we conclude
\bean
|\eta(l)|
&\leq |\eta(l+\tau(l))|+\left|\int_{l}^{l+\tau(l)}\dot{\eta}(s)ds\right|\\
&\leq \overline{c}\left(K+4C\Delta+16C^2+8C^2\|\eta\|_\infty c''+4C+2C\|\eta\|_{\infty}c''+1\right)\\
&\ \ \ +\frac{\|H\|_{\infty}}{\epsilon^2}(2\Delta+8C+4C\|\eta\|_{\infty}c'')  \\
&=\underbrace{\left(8\overline{c}c''C+2\overline cc''+\frac{4c''\|H\|_{\infty}}{\epsilon^2}\right)C}_{=:K'}\|\eta\|_{\infty} \\
&\ \ \ +\underbrace{\overline{c}K+4\overline{c}C\Delta+16\overline{c}C^2+4\overline cC+\overline{c}+\frac{2\|H\|_{\infty}\Delta}{\epsilon^2}+\frac{8C\|H\|_{\infty}}{\epsilon^2}}_{=:K''}.
\eea
Since the above estimate holds for all $l\in\R$
\[
\|\eta\|_{\infty}\leq K'\|\eta\|_{\infty}+K''.
\]
We can achieve that the {\em isoperimetric constant} $C$ satisfies
\beq\label{eqn:Ccond}
C\leq\frac{1}{4} \text{ \ \ and \ \ } K'\leq\frac{1}{2}
\eeq
by choosing $\sigma\in\Om^2(M)$ with small norm.
This proves the proposition.
\end{proof}

\begin{Lemma}\label{lem:awin}

Assume that the isoperimetric constant $C$ is sufficiently small,
then the following holds true.
Suppose that $w=(v,\eta)\in C^\infty(\R\times S^1,T^*N)\times C^\infty(\R,\R)$ is a gradient flow line of the
time dependent gradient $\nabla\A_{\om(s)}$ which converges asymptotically
$\lim_{s\to\pm}w(s)=w_\pm$ to critical points of $\A_{\om(1)},\A_{\om(0)}$ respectively
such that $a=\A_{\om(0)}(w_-)$ and $b=\A_{\om(1)}(w_+)$.
Then the following assertions meet\\
\begin{enumerate}
\item If $a\geq\frac{1}{9}$, then $b\geq\frac{a}{2}$; \\
\item If $b\leq-\frac{1}{9}$, then $a\leq\frac{b}{2}$.
\end{enumerate}
\end{Lemma}

\begin{proof}
By the previous proposition, we obtained the following uniform bound of $\eta$
\bean
\|\eta\|_\infty&\leq2K''\\
&=2\overline{c}K+8\overline{c}C\Delta+32\overline{c}C^2+8\overline cC+2\overline{c}+\frac{4\|H\|_{\infty}\Delta}{\epsilon^2}+\frac{16C\|H\|_{\infty}}{\epsilon^2}.
\eea
Moreover, since $E(w)\geq0$ we obtain from (\ref{eqn:e_bd2}) the inequality
\beqn
b\geq a-4C-2C\|\eta\|_\infty c''.
\eeq
By taking a small isoperimetric constant $C$ satisfying
\bea\label{eqn:Ccond2}
C\overline c c''\leq\frac{1}{32}; \\
C\left(2\overline cC+\frac{\|H\|_\infty}{\epsilon^2}\right)c''\leq\frac{1}{128}; \\
C\left(1+16\overline c c''C^2+4\overline cc''C+\overline c c''+\frac{8c''C\|H\|_\infty}{\epsilon^2}\right)\leq\frac{1}{144};
\eea
we now get
\bea\label{eqn:qaes}
b
&\geq a-4C-2C\|\eta\|_\infty c'' \\
&\geq a-4C-2C\left(2\overline{c}K+8\overline{c}C\Delta+32\overline{c}C^2+8\overline cC+2\overline{c}+\frac{4\|H\|_{\infty}\Delta}{\epsilon^2}+\frac{16C\|H\|_{\infty}}{\epsilon^2}\right) c'' \\
&= a-4C\overline c c''K-8C\left(2\overline cC+\frac{\|H\|_\infty}{\epsilon^2}\right)c''\Delta \\
&\ \; \; \ \ \ -4C\left(1+16\overline c c''C^2+4\overline cc''C+\overline c c''+\frac{8c''C\|H\|_\infty}{\epsilon^2}\right) \\
&\geq a-\frac{1}{8}K-\frac{1}{16}(b-a)-\frac{1}{36},
\eea
where $K=\max\{-a,b\}$.
To prove the assertion (1), we first consider the case
\[
|b|\leq a,\qquad a\geq\frac{1}{9}.
\]
In this case, we estimate
\[
b\geq a-\frac{1}{8}a-\frac{1}{8}a-\frac{1}{36}=\frac{3}{4}a-\frac{1}{36}\geq\frac{a}{2}.
\]
Hence to prove the assertion (1), it suffices to exclude the case
\[
-b\geq a\geq\frac{1}{9}.
\]
But in this case, (\ref{eqn:qaes}) leads to a contradiction in the following way
\[
b\geq\frac{1}{9}+\frac{1}{72}-\frac{1}{16}(b-a)-\frac{1}{36}\geq-\frac{1}{16}(b-a)>0.
\]
This proves the first assertion.
To prove the assertion (2), we set
\[
b'=-a,\qquad a'=-b.
\]
We note that if (\ref{eqn:qaes}) holds for $a$ and $b$, it also holds for $b'$ and $a'$.
Hence we get from the assertion (1) the implication
\[
-b\geq\frac{1}{9} \Longrightarrow -a\geq-\frac{b}{2}
\]
which is equivalent to the assertion (2). This finishes the proof of the Lemma.
\end{proof}

\begin{proof}[{\bf Proof of Theorem \ref{thm:rfhcon}}]
We now construct the continuation homomorphism
\[
\Psi_{\om_0}^{\om_{\sigma}}:\RFHb(\Sigma_k,\om_0=dp\wedge dq)\to\RFHb(\Sigma_k,\om_{\sigma}=dp\wedge dq+\tau^*\sigma)
\]
for $\om_0,\om_{\sigma}\in\Om^{\Mp}_{\rm reg}(\Sigma_k)$.
Similar as in Theorem \ref{thm:continuation},
we first subdivide
\[
\om_s=\om_0+s(\om_\sigma-\om_0)
\]
into small pieces. We first assume that we can find a subdivision
$\{\om^i\}_{i=0}^{N}$ of $\om_s$ satisfying
\begin{itemize}
\item $\om^i=\om_0+d(i)\tau^*\sigma$, where $0=d(0)<d(1)<\cdots <d(N)=1$;
\item $\A_{H,\om^i}:\L\times\R\to\R$ is Morse-Bott, $\forall i=0,1,\dots,N$;
\item $C(M,g,(d(i+1)-d(i))\sigma,v_\nu)$ satisfies (\ref{eqn:Ccond}),\ (\ref{eqn:Ccond2}), $\forall i=0,1,\dots,N-1$.
\end{itemize}
Let $\om_s^i=\om^i+\beta(s)(\om^{i+1}-\om^i)$ be a homotopy between $\om^i$ and $\om^{i+1}$.
First we construct the following continuation map
\[
\widetilde\Psi_{\om^i}^{\om^{i+1}}:\RFHb(\Sigma_k,\om^i)\to\RFHb(\Sigma_k,\om^{i+1}).
\]
Since the action functional $\A_{H,\om^i}$ is Morse-Bott, the
construction is given by counting gradient flow lines with cascades
as in the Morse-Bott homology. Let us choose Morse functions $h^i$
on $\Crit(\A_{H,\om^i})$. We then define a map
\[
\Psi_{\om^i\ \ *}^{\om^{i+1}}:\RFCb_*(\Sigma_k,\om^i)\to\RFCb_*(\Sigma_k,\om^{i+1})
\]
given by
\beqn
\Psi_{\om^i}^{\om^{i+1}}(w_-)=\sum_{\mu(w_+)=\mu(w_-)}\#_2\M_{\om^i}^{\om^{i+1}}(w_-,w_+)w_+,
\eeq
where $w_-\in\Crit(h^i)$, $w_+\in\Crit(h^{i+1})$ and
$\#_2$ denotes the $\Z_2$-counting.
Here,
\bean
\widehat\M_{\om^i,m}^{\om^{i+1}}(w_-,w_+)
&=\{w\ |\ w \text{ is a flow line with }m\text{-cascades from }w_-\text{ to }w_+\};\\
\M_{\om^i,m}^{\om^{i+1}}(w_-,w_+)&=\widehat\M_{\om^i,m}^{\om^{i+1}}(w_-,w_+)/\R^m;\\
\M_{\om^i}^{\om^{i+1}}(w_-,w_+)&=\bigcup_{m\in\N_0}\M_{\om^i,m}^{\om^{i+1}}(w_-,w_+).
\eea

The main issue of this construction is also the uniform bound of
$E(w)$. As in the Morse-Bott homology situation, it suffices to
check that each gradient flow line has a uniform energy bound. For
this reason, we now only consider the following uniform energy
bound. Let
\[
w'=(v',\eta')\in C^{\infty}(\R\times S^1,T^*N)\times C^\infty(\R,\R)
\]
be a gradient flow line of
\[
\A_{\om_s^i}(v,\eta)=\int_Z\overline{v}^*\om_s^i-\eta\int_0^1H(x(t))dt
\]
i.e. a solution of
\beqn\label{eqn:c1} \left.
\begin{array}{cc}
\p_sv+J_{t,s}(v)\left(\p_tv-\eta X_H^{\om_s^i}(v) \right)=0 \\
\p_s\eta+\int_0^1H(x(t))dt=0
\end{array}
\right\} \eeq
\beqn\label{eqn:c2} \lim_{s \to
-\infty}w'(s)=w'_-\in\Crit\A_{\om^i}, \qquad \lim_{s \to
\infty}w'(s)=w'_+\in\Crit \A_{\om^{i+1}}. \eeq To achieve a uniform
energy bound of $w'$, let us recall the equation (\ref{eqn:e_bd}) in
Proposition \ref{prop:lagbd2}
\beqn\label{eqn:ees}
E(w')\leq\A_{\om^{i+1}}(w'_+)-\A_{\om^i}(w'_-)+2CE(w')+4C+2C\|\eta'\|_{\infty}c''.
\eeq
Since the isoperimetric constant $C$ satisfies the condition (\ref{eqn:Ccond}),
we get the following uniform bound of the
Lagrangian multiplier $\eta'$ \beqn\label{eqn:etabd}
\|\eta'\|_{\infty}\leq 2K'' \eeq and \bean
E(w')&\leq2\Delta+8C+4C\|\eta'\|_{\infty}c'' \\
&\leq2\Delta+8C+8Cc''K'', \eea where the coefficients are the same
as in Proposition \ref{prop:lagbd2}. Hence we conclude $E(w')$ is
uniformly bounded.

Now, by virtue of Lemma \ref{lem:awin}, we obtain for $a\leq-\frac{1}{9}$ and $b\geq\frac{1}{9}$ maps
\[
\Psi_{\om_i}^{\om_{i+1}(a,b)}:\RFCb^{(\frac{a}{2},b)}(\Sigma_k,\om_i)\to\RFCb^{(a,\frac{b}{2})}(\Sigma_k,\om_{i+1})
\]
defined by counting gradient flow lines of the time dependent
Rabinowitz action functional. Since the continuation map
$\Psi_{\om_i}^{\om_{i+1}(a,b)}$ commutes with the boundary
operators, this induces the following homomorphism on homology
level.
\[
\widetilde\Psi_{\om_i}^{\om_{i+1}(a,b)}:\RFHb^{(\frac{a}{2},b)}(\Sigma_k,\om_i)\to\RFHb^{(a,\frac{b}{2})}(\Sigma_k,\om_{i+1})
\]
By taking the inverse and direct limit as follows
\[
\RFHb_*(\Sigma_k,\om_i)=\lim_{b\to\infty}\lim_{a\to-\infty}\RFHb_*^{(a,b)}(\Sigma_k,\om_i),
\]
we obtain
\[
\widetilde\Psi_{\om_i}^{\om_{i+1}}:\RFHb(\Sigma_k,\om_i)\to\RFHb(\Sigma_k,\om_{i+1}).
\]
Similar in usual Floer homology, we can define the continuation
homomorphism by juxtaposition
\[
\widetilde\Psi_{\om_0}^{\om_\sigma}:\RFHb(\Sigma_k,\om_0)\to\RFHb(\Sigma_k,\om_\sigma)
\]
\[
\widetilde\Psi_{\om_0}^{\om_\sigma}=\widetilde\Psi_{\om^{N-1}}^{\om^N}\circ\cdots\circ\widetilde\Psi_{\om^1}^{\om^2}\circ\widetilde\Psi_{\om^0}^{\om^1}.
\]
In a similar way, we can construct \beqn
\widetilde\Psi_{\om_\sigma}^{\om_0}:\RFHb(\Sigma_k,\om_\sigma)\to\RFHb(\Sigma_k,\om_0),
\eeq by following the homotopy backwards. By a
homotopy-of-homotopies argument, we conclude
$\widetilde\Psi_{\om_\sigma}^{\om_0}\circ\widetilde\Psi_{\om_0}^{\om_\sigma}=\id_{\RFHb(\Sigma_k,\om_0)}$
and
$\widetilde\Psi^{\om_\sigma}_{\om_0}\circ\widetilde\Psi^{\om_0}_{\om_\sigma}=\id_{\RFHb(\Sigma_k,\om_\sigma)}$.
Therefore $\widetilde\Psi_{\om_0}^{\om_1}$ is an isomorphism with
inverse $\widetilde\Psi^{\om_0}_{\om_1}$.

It remains to discuss the case where the corresponding subdivision
$\{\omega^i\}_{i=0}^N$ of $\omega_s$ does not exist. The issue is
the assertion about Morse-Bott which in higher dimensions is not
known to hold for generic choice of the magnetic field. In this case
we can perturb Rabinowitz action functional by an additional
non-physical perturbation already used in \cite{CFP09}. Namely
choose a compactly supported time-dependent Hamiltonian $F \in
C^\infty(T^*N \times S^1)$ and consider the perturbed Rabinowitz
action functional $\mathcal{A}_\omega^F \in
\mathscr{L}\times\mathbb{R} \to \mathbb{R}$ defined by
$$\mathcal{A}_\omega^F(v,\eta)=\mathcal{A}_\omega(v,\eta)+\int_0^1F(v(t),t)dt.$$
For generic perturbation $F$ the perturbed Rabinowitz action
functional is Morse, see \cite{CFP09}. Moreover, the difference
$\mathcal{A}_\omega-\mathcal{A}_\omega^F$ is uniformly bounded by
the Hofer norm of the perturbation $F$. By choosing a small enough
perturbation all our previous estimates hold up to an arbitrarily
small error term. This procedure allows us to construct a
continuation homomorphism between the two Rabinowitz Floer
homologies in the unlikely case where we cannot directly interpolate
between the two symplectic forms.
\end{proof}


\appendix

\section{Cofilling function in Example \ref{ex:sol}}\label{app:sol}

In this appendix, we will give more detailed explanation about Example \ref{ex:sol}.
We follow the idea of geometric group theory listed in \cite{ECH}.

\begin{proof}[Proof of Example \ref{ex:sol}]
Let us recall the 3-manifold $M$ in Example \ref{ex:sol}, fibered over $S^1$ with fiber $\T^2$ with hyperbolic monodromy
\[
A=\begin{pmatrix}
2&1\\
1&1
\end{pmatrix}.
\]
Let $y,z$ be the coordinates of the fiber torus, then $\sigma =dy \wedge dz$ is a well-defined 2-form on $M$.
In order to show the exponential growth of the cofilling function $u_\sigma$,
we now consider the group $G$ generated by the following action on $\R^3$:
\bean
\alpha :&\  (x,y,z)\mapsto \big(x+1,A(y,z)\big); \\
\beta  :&\  (x,y,z)\mapsto (x,y+1,z); \\
\gamma :&\  (x,y,z)\mapsto (x,y,z+1).
\eea

Note that the quotient space $G\setminus\R^3$ is the manifold $M$
with the universal covering map $p:\R^3\to M$.
Since there exists a Riemannian metric $g$ on $M$,
we get the pullback metric $\widetilde g=p^*g$ on $\R^3=\widetilde M$
which is invariant under the action of the group $G$.
Especially we choose a metric $g$ satisfying the condition that $\widetilde g$
has length 1 for the following edges in $\widetilde M$,
\bean
(x,y,z)&\sim\big(x+1,A(y,z)\big); \\
(x,y,z)&\sim(x,y+1,z); \\
(x,y,z)&\sim(x,y,z+1);
\eea
where $p\sim q$ means the straight line connecting $p$ and $q$ in Euclidean metric.

Consider the word
\[
w_n=(\alpha^n\gamma^{-1}\alpha^{-n})(\alpha^{-n}\beta^{-1}\alpha^n)(\alpha^n\gamma\alpha^{-n})(\alpha^{-n}\beta\alpha^n)
\]
which represents the identity.
Then $w_n$ is regarded as a contractible path in the universal cover $(\R^3,\widetilde g)$,
travels around the following points with straight lines:
\[
(0,0,0),\ (n,0,0),\ (n,e_1),\ (0,A^{-n} e_1),\ (-n,A^{-2n} e_1),
\]
\[
(-n,A^{-2n} e_1+e_2),\ (0,A^{-n} e_1+A^n e_2),\ (n,e_1+A^{2n} e_2),
\]
\[
(n,A^{2n} e_2),\ (0,A^{n} e_2),\ (-n,e_2),\ (-n,0,0),\ (0,0,0).
\]
Note that the length of $w_n$ grows linearly as $n\to\infty$ \wrt the metric $\widetilde g$.

We claim that any bounding disk $D_n$ with $\partial D_n=w_n$
has area at least $k\lambda^{2n}$,
where $k$ is a positive constant and $\lambda$ is the eigenvalue of $A$ bigger than 1.
To show this, consider the following projection
\bean
\pi:\R^3\to\R^2,\qquad
(x,y,z)\mapsto(y,z).
\eea
Let $\sigma'$ be the area form $dy\wedge dz\in\Om^2(\R^2)$,
then we get $p^*\sigma=\pi^*\sigma'$.
Note that this form is preserved by the action of $G$ and we obtain
\[
\int_{p(D_n)}\sigma=\int_{D_n}p^*\sigma=\int_{D_n}\pi^*\sigma'=\int_{\pi(D_n)}\sigma'.
\]
The projection $\pi(D_n)$ contains the parallelogram with vertices
$0, A^{-n} e_1, A^{-n} e_1+A^n e_2$ and $A^n e_2$ in the $yz$-plane,
and the area of this one is approximately $k\lambda^{2n}$ for large $n$.

Now suppose that there exists a primitive $\theta\in\Om^1(\widetilde M)$ of $p^*\sigma$
with subexponential growth $f(n)=\sup_{z\in B_0(n)}|\theta_z|_{\widetilde p}$.
Then we deduce the following contradiction,
\[
\epsilon\lambda^{2n}\leq\int_{p(D_n)}\sigma=\int_{D_n}p^*\sigma=\int_{\partial D_n}\theta
\leq f(4n+2)\int_{w_n}1=f(4n+2)\cdot (8n+4)
\]
as $n\to\infty$ for small enough $\epsilon>0$.
If we take $ydz\in\Om^1(\widetilde M)$ as a primitive of $p^*\sigma$
then by direct calculation $\int_{p_n}ydz\leq K\lambda^{2n}$ as $n\to\infty$,
where $p_n:[0,n]\to\widetilde M$ is a length $n$ path with $p(0)=0$ and $K$ large enough.
This implies that $\sup_{q\in B_0(n)}|ydz_q|_{\widetilde p}$ has at most exponential growth and
we conclude that $u_{\sigma}(s)\sim \exp(s)$.

\end{proof}

%
%
\end{document}